\documentclass[letterpaper, 10 pt, conference]{ieeeconf}  % Comment this line out if you need a4paper
\usepackage[a-1b]{pdfx}

\IEEEoverridecommandlockouts                              % This command is only needed if 
\usepackage{graphics} % for pdf, bitmapped graphics files
\graphicspath{{cdc_2018_figures/}}
\usepackage{amsmath} % assumes amsmath package installed
\usepackage{amssymb}  % assumes amsmath package installed
\usepackage{verbatim}

\newcounter{rmnum}
\newenvironment{romannum}{\begin{list}{{\upshape (\roman{rmnum})}}{\usecounter{rmnum}
\setlength{\leftmargin}{4pt}
\setlength{\rightmargin}{4pt}
\setlength{\itemsep}{1pt}
\setlength{\itemindent}{5pt}
}}{\end{list}}

\def\Ebox#1#2{%
\begin{center}
\includegraphics[width= #1\hsize]{#2}\end{center}}

\def\rd#1{{\color{red}#1}}
\def\spm#1{{\color{red}spm: #1}}

\def\epsy{\varepsilon}
\def\eqdef{\mathbin{:=}}

\def\limsup{\mathop{\rm lim\ sup}}

\def\argmin{\mathop{\rm arg\, min}}

\def\transpose{{\hbox{\it\tiny T}}}

\def\IEEEQEDclosed{\mbox{\rule[0pt]{1.3ex}{1.3ex}}} % for a filled box
\def\qed{\hspace*{\fill}~\IEEEQEDclosed\par\endtrivlist\unskip}%\mbox{\rule[0pt]{1.3ex}{1.3ex}}}

\newtheorem{theorem}{Theorem}[section]

\newtheorem{proposition}[theorem]{Proposition}
\newtheorem{lemma}[theorem]{Lemma}

\newcommand{\field}[1]{\mathbb{#1}}

\def\Re{\field{R}}

\def\barf{{\overline {f}}}
\def\barh{{\overline {h}}}
\def\barsigma{\overline{\sigma}}
\def\barSigma{\overline{\Sigma}}
\def\uQ{\underline{Q}}
\def\uF{\underline{F}}
\def\bfmath#1{{\mathchoice{\mbox{\boldmath$#1$}}%
{\mbox{\boldmath$#1$}}%
{\mbox{\boldmath$\scriptstyle#1$}}%
{\mbox{\boldmath$\scriptscriptstyle#1$}}}}
 
 \def\bfma{\bfmath{a}}

\def\bfmx{\bfmath{x}}

\def\bfmu{\bfmath{u}}

\def\bfmz{\bfmath{z}}

\def\bfmM{\bfmath{M}}

\def\bfmW{\bfmath{W}}

 \def\FRAC#1#2#3{\genfrac{}{}{}{#1}{#2}{#3}}

\def\ddt{{\mathchoice{\FRAC{1}{d}{dt}}%
{\FRAC{1}{d}{dt}}%
{\FRAC{3}{d}{dt}}%
{\FRAC{3}{d}{dt}}}}

\def\ddr{{\mathchoice{\FRAC{1}{d}{dr}}%
{\FRAC{1}{d}{dr}}%
{\FRAC{3}{d}{dr}}%
{\FRAC{3}{d}{dr}}}}

\def\ddr{{\mathchoice{\FRAC{1}{d}{dy}}%
{\FRAC{1}{d}{dy}}%
{\FRAC{3}{d}{dy}}%
{\FRAC{3}{d}{dy}}}}

\def\half{{\mathchoice{\FRAC{1}{1}{2}}%
{\FRAC{1}{1}{2}}%
{\FRAC{3}{1}{2}}%
{\FRAC{3}{1}{2}}}}

\def\clE{{\cal E}}

\def\Expect{{\sf E}}

\usepackage{wrapfig}

\usepackage{graphicx}

\graphicspath{{figures/}}

\usepackage{caption}
\usepackage{subcaption}

\usepackage{float}

\usepackage{tikz}
\usetikzlibrary{calc,shapes, shapes.geometric, shapes.symbols, shapes.arrows, shapes.multipart, shapes.callouts, shapes.misc,arrows.meta}

\usepackage{pgfplots}

\def\urls#1{{\small \url{#1}}}

\def\rd#1{{\color{red}#1}}

\def\Lemma#1{Lemma~\ref{#1}}
\def\Prop#1{Prop.~\ref{#1}}
\def\Theorem#1{Theorem~\ref{#1}}

\def\Section#1{Section~\ref{#1}}

\def\Fig#1{Fig.~\ref{#1}}

\newlength{\noteWidth}
\long\def\notes#1{\ifinner
             {\tiny #1}
             \else
              \marginpar{\parbox[t]{\noteWidth}{\raggedright\tiny #1}}
               \fi}

\def\notes#1{\typeout{#1 !!!}}  % for final copy

\def\spm#1{\notes{SPM:  #1}}

\def\fb{\phi}

\def\barA{\bar A}
\def\barlambda{{\bar{\lambda}}}

\def\bfalpha{\bfmath{\alpha}}

\def\Real{\text{Re}}

\def\tilnu{\tilde{\nu}}

\def\tilf{\tilde f}

\def\bfgamma{\bfmath{\gamma}}
\def\bfxi{\bfmath{\xi}}

\def\bfnu{\bfmath{\nu}}

\def\half{{\mathchoice{\FRAC{1}{1}{2}}%
{\FRAC{1}{1}{2}}%
{\FRAC{3}{1}{2}}%
{\FRAC{3}{1}{2}}}}

\def\ddt{{\mathchoice{\FRAC{1}{d}{dt}}%
{\FRAC{1}{d}{dt}}%
{\FRAC{3}{d}{dt}}%
{\FRAC{3}{d}{dt}}}}

\def\ddr{{\mathchoice{\FRAC{1}{d}{dr}}%
{\FRAC{1}{d}{dr}}%https://www.overleaf.com/project/5c4612ddbae4140d6a60530f
{\FRAC{3}{d}{dr}}%
{\FRAC{3}{d}{dr}}}}

\def\ddu{{\mathchoice{\FRAC{1}{d}{du}}%
{\FRAC{1}{d}{du}}%
{\FRAC{3}{d}{du}}%
{\FRAC{3}{d}{du}}}}

\def\ddw{{\mathchoice{\FRAC{1}{d}{dw}}%
{\FRAC{1}{d}{dw}}%
{\FRAC{3}{d}{dw}}%
{\FRAC{3}{d}{dw}}}}

\def\haG{\widehat{G}}

\newcommand{\varqsa}{\theta}
\newcommand{\varode}{\raisebox{.15em}{\mbox{$\chi$}}}
\newcommand{\varscaled}{\widehat{\varode}}
\def\tilvarqsa{\tilde\varqsa}

\newcommand{\bfvarqsa}{\bfmath{\varqsa}}
\newcommand{\bfvarscaled}{\bfmath{\varscaled}}
\newcommand{\bfvarode}{\bfmath{\varode}}

\def\sfT{{\hbox{\tiny \textsf{T}}}}

 \title{\bf Optimal Rate of Convergence for Quasi-Stochastic Approximation
 }
\IEEEoverridecommandlockouts

\author{Andrey Bernstein\authorrefmark{1} \and Yue Chen\authorrefmark{1}  \and  Marcello Colombino\authorrefmark{1}   \and   Emiliano Dall'Anese\authorrefmark{3}  \and Prashant Mehta\authorrefmark{4} 
\and Sean Meyn\authorrefmark{5}%
\thanks{Authors are in alphabetical order.}
\thanks{\authorrefmark{1}A.B.,\, M.C.,\, and Y.C.\ are  with   NREL in Golden Colorado
({\tt\small  name.lastname@nrel.gov})}%
\thanks{\authorrefmark{3}E.D. is the Department of ECEE at the University of Colorado Boulder
({\tt\small Emiliano.Dallanese@colorado.edu})}%
\thanks{\authorrefmark{4}P.G.M. is with the Department of MAE at the University of Illinois Urbana-Champaign
({\tt\small mehtapg@illinois.edu})}%
\thanks{\authorrefmark{5}S.M. is with the Department of ECE at the University of Florida in Gainesville
({\tt\small meyn@ece.ufl.edu})}%
\thanks{\textbf{Acknowledgements}:
 Financial support from NSF CMMI grant 146277,  NSF CPS grant  1646229, and ARO grant W911NF1810334 is gratefully acknowledged. This work was authored in part by the National Renewable Energy Laboratory,
managed and operated by Alliance for Sustainable Energy, LLC, for
the U.S. Department of Energy (DOE) under Contract No. DE-AC36-08GO28308. This work was supported in part by the Laboratory Directed Research and Development (LDRD) Program at NREL. The views expressed in
the article do not necessarily represent the views of the DOE or the
U.S. Government. The U.S. Government retains and the publisher, by
accepting the article for publication, acknowledges that the U.S. Government
retains a nonexclusive, paid-up, irrevocable, worldwide license to publish
or reproduce the published form of this work, or allow others to do so, for
U.S. Government purposes.}
}

\begin{document}

\maketitle
\thispagestyle{empty}
%\pagestyle{empty}

%%%%%%%%%%%%%%%%%%%%%%%%%%%%%%%%%%%%%%%%%%%%%%%%%%%%%%%%%%%%%%%%%%%%%%%%%%%%%%%%

\begin{abstract}  
The  Robbins-Monro stochastic approximation algorithm is a foundation of many algorithmic frameworks for reinforcement learning (RL), and  often an efficient approach to solving (or approximating the solution to) complex  optimal control problems.  However, in many cases practitioners are unable to apply these techniques because of an inherent high variance. This paper aims to provide a general foundation for ``quasi-stochastic approximation,'' in which all of the processes under consideration are deterministic, much like quasi-Monte-Carlo for variance reduction in simulation. The variance reduction can be substantial, subject to tuning of pertinent  parameters in the algorithm.  This paper introduces a new coupling argument to establish optimal rate of convergence provided the gain is sufficiently large.  These  results are established for linear models, and tested also in non-ideal settings.

A major application of these general results is a new class of RL algorithms for deterministic state space models. In this setting, the main contribution is a class of algorithms for approximating the value function for a given policy,  using a different policy designed to introduce exploration.

%https://www.overleaf.com/project/5c4612ddbae4140d6a60530f

\end{abstract}

% \subsection*{Glossary  (currently for our consumption)}

% \begin{romannum}

% \item
% $\varqsa(t)$   Parameter at time $t$  \qquad  \verb+\varqsa(t)+

% \item 
% $\varscaled(u)
%  =  \theta(g(u))   $   \eqref{e:varscaled} \qquad  \verb+\varscaled(u)+

% \item 
% $\varode(w) $    Solution to the ODE \eqref{eq:avgODE} \qquad  \verb+\varode(w)+

% \item
% $\varode^u(w)$ 
% is the solution to \eqref{eq:avgODE} satisfying $ \varode^u(u) = \varscaled(u)  $  
% \eqref{eq:init}
 
%  \item $a(t)$ stepsize \eqref{eq:algo}

% \item
% Bold font for processes, such as
% $
% \bfvarqsa  =  \{ \varqsa(t) : t\ge 0 \}$.

% \bl{SM2ED  a convention since Meyn \&\ Tweedie -- bold for a sequence or process.   We can't use $W$ since this is a generic random variable, and $\{W_i : i\ge 0\}$ is too bulky.
% \\
% And I like my ``$dt$'' after integrals -- whoever is deleting, please explain or stop!!}

%\end{romannum}

%%%%%%%%%%%%%%%%%%%
\section{Introduction and Proposed Framework}
\label{s:intro}
%%%%%%%%%%%%%%%%%%%

Stochastic approximation concerns the root-finding problem $\barf(\varqsa^*) = 0$, where  $\varqsa^*\in\Re^d$ is a parameter to be computed or approximated, and $\barf: \Re^d  \to \Re^d$ is defined using the following expectation
\begin{equation}
\barf(\varqsa) := \Expect[f(\varqsa,\xi)]\,,\qquad \varqsa\in\Re^d\,, 
\label{e:barf}
\end{equation}
in which $f: \Re^d \times \Re^m \to \Re^d$ and $\xi$ is an $m$-dimensional random vector. With this problem in mind, the stochastic approximation (SA) method of Robbins and Monro~\cite{robmon51a,bor08a} involves recursive algorithms to estimate the parameter $\varqsa^*$. The simplest algorithm is defined by the following recursion ($n$ is the iteration index): 
\begin{equation} 
\varqsa_{n+1} = \varqsa_n + a_n f(\varqsa_n,\xi_n)\, , \qquad n\ge 0,
\label{eq:SA}
\end{equation}
where  $\bfxi :=\{\xi_n \}$ is an exogenous $m$-dimensional stochastic process, $a_n > 0$ is the step size, and $\varqsa_0\in\Re^d$ is given.   For consistency with \eqref{e:barf}, it is assumed that the distribution of $\xi_n$ converges to that of $\xi$ as $n\to\infty$; e.g.,   $\bfxi$ is an ergodic Markov process.   

The motivation for the SA recursion (and also an important tool for convergence analysis) is the associated ordinary differential equation (ODE):
\begin{equation}
\ddu \varode (u)  = \barf \left( \varode (u) \right).
\label{eq:avgODE}
\end{equation}   
Under general assumptions, including boundedness of  the stochastic recursion \eqref{eq:SA}, the
limit points of \eqref{eq:SA} are a subset of the stationary points of the ODE; that is, solutions to $\barf(\varqsa^*) = 0$.  See~\cite{bor08a,bormey00a} and the earlier monographs \cite{benmetpri90,kusyin97}.

The upshot of stochastic approximation is that it can be implemented  without knowledge of the function $f$ or of the distribution of $\xi$; rather, it can rely on observations of the sequence $\{ f(\varqsa_n,\xi_n) \}$. This is one reason why these algorithms are valuable in the context of reinforcement learning (RL)~\cite{bor08a,bertsi96a,huachemehmeysur11,devmey17a,devmey17b}.
In such cases, the driving noise is typically modeled as a Markov chain.

To introduce the proposed framework, a key observation is that Markov chains need not be stochastic; for example, for given $\omega>0$, the sequence $\xi_n=(\cos(\omega n), \sin(\omega n)$  is a Markov chain on the unit circle in $\Re^2$. This motivates us to consider the following variant of algorithm \eqref{eq:SA}:
%It is convenient to develop the theory in a continuous time setting.    
%\spm{2019.  Is it useful to reference ML literature saying that it makes sense to develop algorithms in continuous time?}
%The
%\emph{quasi-stochastic approximation} (QSA) algorithm considered in this paper is defined by the ODE 
\begin{equation}
\ddt\varqsa(t) = a(t)f(\varqsa(t),\xi(t)) \,,
\label{eq:algo}
\end{equation}
where the ``noise'' $\bfxi$ is generated from a \emph{ deterministic} (possibly oscillatory) signal rather than a stochastic process.  We term this iteration a \emph{quasi-stochastic approximation} (QSA) algorithm\footnote{The-continuous time setting is adopted mainly for simplicity of exposition, especially for the convergence analysis; results  can be extended to the discrete-time setting, but are omitted for space constraints.}.  

One motivation for the proposed framework  was to   provide foundations for the Q-learning algorithm  introduced in~\cite{mehmey09a}, which treats nonlinear optimal control in continuous time. In \cite{mehmey09a} it was found in numerical experiments that the rate of convergence is superior to the ones of  traditional applications of Q-learning.  The present paper provides explanations for this fast convergence,  and presents a methodology to design algorithms with optimal rate of convergence.

\subsection{Contributions}

Contributions of the present paper are explained in terms of theoretical advancements for the QSA and applications.  

\paragraph{Analysis}
As in the  classical SA algorithm,   analysis is based on consideration of the associated  ODE \eqref{eq:avgODE} in which the ``averaged'' vector field is given by the ergodic average:  
\begin{equation}
\barf(\varqsa) = \lim_{T\rightarrow\infty}\frac{1}{T}\int_0^T f(\varqsa,\xi(t))\, dt,\ \ \textrm{for all }\varqsa\in\Re^d.
\label{eq:ergodic}
\end{equation}
The paper will introduce pertinent assumptions   in \Section{s:QSA} to ensure that the limit 
\eqref{eq:ergodic} exists, and that the averaged ODE  \eqref{eq:avgODE} has a unique globally asymptotically stable stationary point $\varqsa^*$. It will  be shown that the QSA \eqref{eq:algo} converges to the same limit. Relative to convergence theory  in the stochastic setting, new  results concerning rates of convergence will be offered in \Section{s:QSA}.     
 
The variance analysis outlined in \Section{s:QSA} begins by considering a linear setting $\bar{f}(\varqsa) = A(\varqsa-\varqsa^*)$, with $A$ Hurwitz. The linearity assumption is typical in much of the literature on variance for stochastic approximation  and is justified by constructing a linearized approximation for the original nonlinear algorithm \cite{kontsi04,kusyin97}.  Rates of convergence of nonlinear QSA is beyond the scope of this paper and will be pursued in future work.   

%Under additional mild assumptions on $\bfxi$ it is shown that the rate of convergence is   \emph{sub-linear}; that is, $ t^{\varrho} \|\varqsa(t)  - \varqsa^* \|   $ is a bounded function of time,  for some $\varrho\in(0,1]$.  A bound on $\varrho$ is obtained based on the eigenvalues of the matrix $A$.  \spm{to do!}
Under the assumption   that $I+A$ is Hurwitz (that is, each eigenvalue $\lambda$ of $A$ satisfies $\Real(\lambda) < -1$), it will be shown that the optimal rate of convergence of $1/t$ can be obtained.  In particular, 
there is a constant $\barsigma<\infty$ such that the following holds for each initial condition $\varqsa(0)$:
\begin{equation}
\limsup_{t\to\infty} \, t \|\varqsa(t)  - \varqsa^* \|  \le \barsigma 
\label{e:QSArate}
\end{equation} 
This assumption is stronger than what is imposed to obtain the Central Limit Theorem for stochastic approximation, which requires $\Real(\lambda) < -\half$. On the other hand, the conclusions for stochastic approximation algorithms are weaker, where the above bound is replaced by
\begin{equation}
\limsup_{t\to\infty} \, t \Expect[\|\varqsa(t)  - \varqsa^* \|^2 ] \le \barsigma^2 
\label{e:avar}
\end{equation}
That is, the rate is $1/\sqrt{t}$ rather than $1/t$
\cite{benmetpri90,kusyin97}.

\paragraph{Applications}
The most compelling applications are:   (i) gradient-free  optimization methods,  based on ideas from extremum-seeking control  \cite{liukrs12,arikrs03};  and (ii) RL for deterministic control systems.   Q-learning with function approximation is reviewed, following \cite{mehmey09a}. % The paper contains a brief exposition on how QSA can be applied in optimization, and   more attention to applications in RL.
It is shown that the most straightforward application of RL does not satisfy the conditions of the paper, and in fact may not be stable.  In view of these challenges,  a new class of ``off policy''  RL algorithms are introduced.   These algorithms  have attractive numerical properties, and are suitable for application to approximate policy iteration.  

\spm{The new approach is illustrated with application to  \rd{do be written}}

\subsection{Literature review}

%For background on quasi-Monte Carlo there are many books and recent proceedings, such as \cite{leopil14,glyowe16}.   There is a significant history of application in finance, and in fact the 
The first appearance of QSA methods appears to have originated in the domain of quasi-Monte Carlo methods applied to finance; see~\cite{lappagsab90,larpag12}.  Rates of convergence were obtained in \cite{shimey11}, but with only partial proofs, and without the coupling bounds reported here.

Gradient-free optimization has been studied in two, seemingly disconnected lines of work. The first line of work, typically known as ``bandit optimization'' (see e.g., \cite{Flaxman2005,Awerbuch2008,Bubeck2012}) leverages a \emph{stochastic} estimate of the gradient, based on a single or multiple evaluations of the objective function. Such algorithms have been analyzed extensively using tools similar to the classical SA approach, with similar conclusion on the high variance of the estimates \cite{Chen2019}. The second line of work, typically termed ``extremum-seeking control'' (ESC) \cite{liukrs12,arikrs03}, leverages a \emph{deterministic} estimate of the gradient; it is, in fact, a special application of the QSA theory developed in this paper. Stability of the classic ESC feedback scheme was analyzed in e.g., \cite{Krstic2000,HHW2000};  see \cite{arikrs03} 
for a comprehensive overview of the methods. %Sinusoids are the most popular perturbation signals in the ESC literature, leading to fast convergence rates in many applications, compared to stochastic perturbations. 
%In this paper, we hope to shed some light on the difference between these two approaches and on the relative success of the deterministic variants of gradient-free optimization.

%The existing literature on extremum-seeking control confirms numerically very fast convergence rates with small variance; however, they typically lack theoretical analysis. In the current paper, we hope to shed some light on the difference between these two approaches and on the relative success of the deterministic variants of gradient-free optimization.
%\andrey{Yue: can you please add more references/details on ES? \\
%The extremum seeking control was firstly popular in the 1940s - 1960s, focused on new algorithm development and performance evaluation. There was little progress on theory until 2000, when Wang and Krstics published stability analysis works on the classic extremum seeking feedback scheme \cite{Krstic2000,HHW2000}. Since then, extremum seeking control has been well studied and \cite{arikrs03} summarizes the comprehensive analysis of the extremum seeking control and its successful applications on various fields. Sinusoids are the most popular perturbations in the extremum seeking control and are shown to be fast convergence in many applications, in compare to stochastic perturbations.}

The rate of convergence result \eqref{e:avar} is an interpretation of classical results in the SA literature.   Under mild conditions, the ``limsup'' can be replaced by a limit,  and moreover the Central Limit Theorem  holds for the scaled error process $\{ \sqrt{t}[\varqsa(t)  - \varqsa^*]\}$ \cite{benmetpri90,kusyin97,bor08a}.  In these works, the  asymptotic covariance is the solution to a Lyapunov equation,  derived from the linearized ODE   and the noise covariance.    The results in the QSA setting  are  different. It is shown in  \Theorem{t:var} that under the Hurwitz assumption on $I+A$,  the scaled parameter estimates $\{  {t}[\varqsa(t)  - \varqsa^*]\}$  \textit{couple} with another process, obtained by integrating the noise process.   There is a large literature on techniques to minimize the  asymptotic variance in stochastic approximation, including Ruppert-Polyak-Juditsky (RPJ) averaging \cite{rup88,poljud92},  or adaptive gain selection,  resulting in the stochastic Newton-Raphson (SNR) algorithm  \cite{rup85,kusyin97}.  % See \cite{devmey17a,devmey17b} for a recent tutorial, with application to Q-learning.   
The problem of optimizing the rate for QSA (e.g.,  minimizing the bound $\barsigma$ in \eqref{e:QSArate}) through choice of algorithm parameters is not trivial.   This is because coupling occurs only when the eigenvalues of $A$ satisfy  $\Real(\lambda) < -1$.   If one eigenvalue reaches the lower bound, so that $\Real(\lambda_0) = -1$, then the theory presented here predicts that either $\barsigma = \infty$ or it is an unbounded function of the initial condition $\varqsa(0)$.

The fixed-policy Q-learning algorithm introduced here may be regarded as an \textit{off policy} TD-learning algorithm  (or SARSA) \cite{sze10,sutbar98}. The standard TD and SARSA algorithms are not well-suited to deterministic systems since the introduction of exploration creates bias.
By definition, an off policy method allows an arbitrary stable input, which can be chosen to speed value function estimation.   Q-learning also allows for exploration, but this is a nonlinear algorithm that often presents numerical challenges,  and there is little theory to support this class of algorithms beyond special cases such as optimal stopping, or the complex ``tabular'' case for finite state-space models \cite{sze10,sutbar98}. In the special case of linear systems with quadratic cost,    the off-policy TD learning algorithm introduced here reduces to \cite{braydsbar94}.

 \paragraph*{Organization}
The remainder of this paper is organized as follows.   
Sections \ref{s:appl} and \ref{sec:RL} contain several general application areas for QSA, along with numerical examples.   Stability and convergence theory is summarized in \Section{s:QSA}, with most technical proofs contained in the Appendix.  
Conclusions and future directions for research are summarized in
\Section{s:conclude}.

%\section{Applications and Numerical Examples}
\section{Motivational Application Examples}
\label{s:appl} 

%\subsection{Basic algorithm}

To motivate the QSA theory, this section  briefly discusses quasi Monte-Carlo and gradient-free optimization. A deeper look  at applications to optimal control, which is the main focus of this paper, will be given in Section \ref{sec:RL}.

% \begin{figure}[ht]
% \Ebox{.75}{OneMCplot.pdf} 
% \vspace{-1em}
% \caption{Sample paths of estimates obtained using Monte-Carlo and Quasi Monte-Carlo} 
% \label{f:OneMCplot}
% \end{figure}

%This section contains examples to  motivate the theory; a deeper look at applications to optimal control is po; stability and convergence analysis of the general algorithm \eqref{eq:algo} is postponed to %\Section{s:QSA}.   

 %In all the subsequent examples, it is assumed that the process $\bfxi$ appearing in  \eqref{eq:algo} is constructed so that it is deterministic, but \emph{ergodic}, in the sense that the limit in \eqref{eq:ergodic} exists.  Typically we construct $\bfxi$ as a mixture of sinusoids, or other periodic signals.  In such cases, continuity of $f$ is sufficient to obtain  \eqref{eq:ergodic}.  

\subsection{Quasi Monte-Carlo}
\label{s:QMC}

Consider the problem of obtaining the integral over the interval $[0,1]$ of a function $y\colon\Re\to\Re$.  
To fit the QSA model \eqref{eq:algo}, let $\xi(t) := t $ (modulo 1),  and set
\begin{equation}
f(\varqsa, \xi) := y(\xi) - \varqsa.
\end{equation}
The averaged function is then given by   
\begin{align*}
\barf (\varqsa) =
 \lim_{T\rightarrow\infty}\frac{1}{T}\int_0^T f(\varqsa,\xi(t))\, dt 
 = \int_0^1 y(t)\, dt  - \varqsa  
\end{align*}
so that  $\varqsa^* = \int_0^1 y(t) \, dt $.
Algorithm \eqref{eq:algo} is given by:
\begin{equation}
 \ddt \varqsa(t) = a(t) [y(\xi(t)) - \varqsa(t) ].
 \label{e:QMC}
\end{equation}

The numerical results that follow are based on the function $ y(t) =  e^{4t}\sin(100 t) $. This exotic function was among many tested -- it is used here only because the conclusions are particularly striking.

\begin{figure}[ht] 
\Ebox{1}{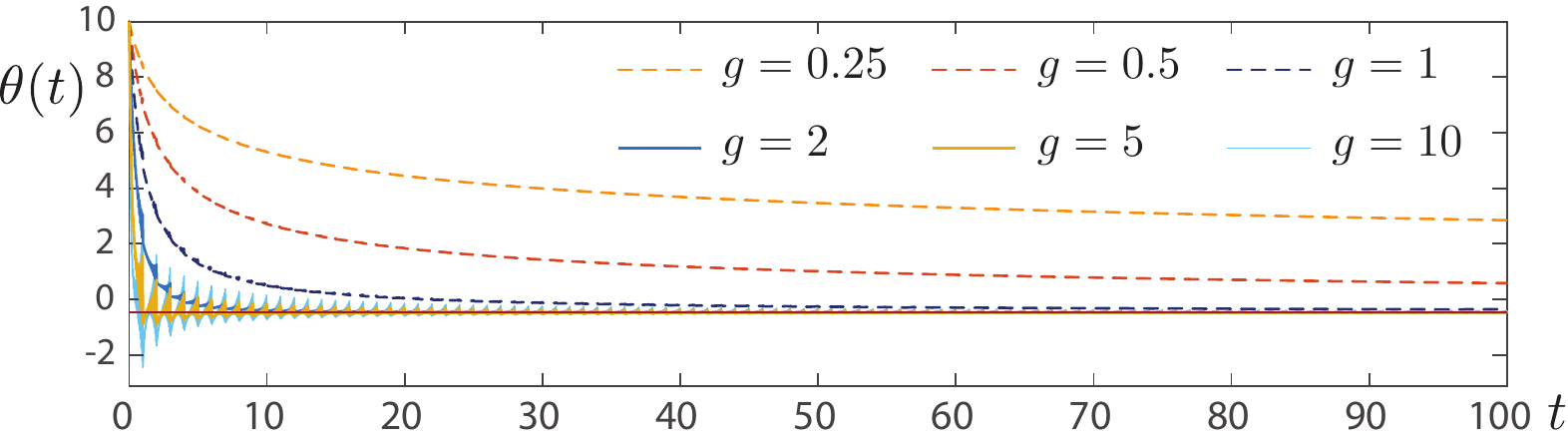} 
\caption{\small Sample paths of Quasi Monte-Carlo estimates.} 
\vspace{-.5em}
\label{f:plotsQMC}
\end{figure}

The differential equation was approximated using a standard Euler scheme with step-size $10^{-3}$.    
Two algorithms are compared in the numerical results that follow:  standard Monte-Carlo,   and versions of the deterministic algorithm \eqref{e:QMC}, differentiated by the gain $a(t)=g/(t+1)$.  
\Fig{f:plotsQMC}  shows typical sample paths of the resulting estimates for a range of gains;  in each case the algorithm was initialized with $\varqsa(0)=10$.  
The true mean is $\varqsa^* \approx   -0.4841$.

Independent trials were conducted to obtain variance estimates.    In each of $10^4$ independent runs, the common initial condition was drawn from $N(0,10)$,  and the estimate was collected at time $T=100$. 
\Fig{f:hists}  shows three histograms of estimates for standard Monte-Carlo, and QSA using gains $g=1$ and $2$.   
An alert reader must wonder:   \textit{why is the variance reduced by 4 orders of magnitude when the gain is increased from $1$ to $2$?}  The relative success of the high-gain algorithm is explained in  \Section{s:QSA}.

\begin{figure}[ht]
\Ebox{.99}{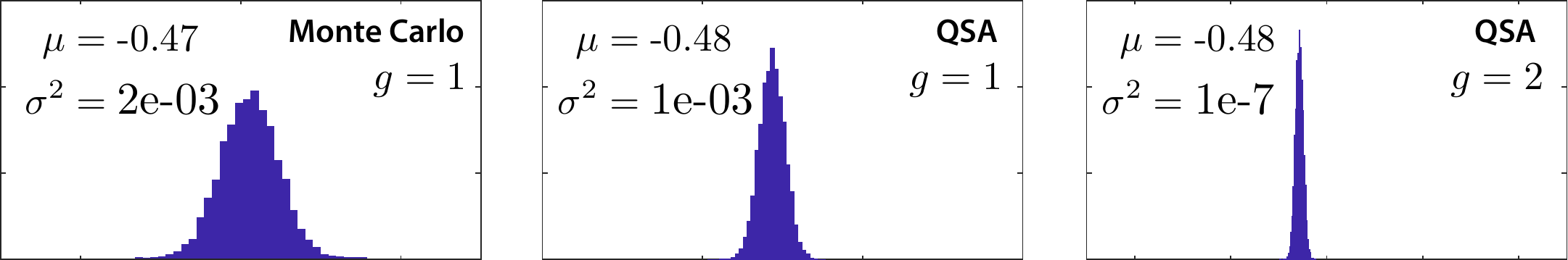} 
\caption{\small Histograms of Monte-Carlo and Quasi Monte-Carlo estimates after $10^4$ independent runs.} 
\vspace{-1em}
\label{f:hists}
\end{figure}

\def\probe{\xi}
\def\bfprobe{\bfxi}

\subsection{Gradient-Free Optimization} \label{sec:grad_free}

Consider the unconstrained convex minimization problem
\begin{equation}
\min_{\varqsa \in \Re^d} J(\varqsa).
\end{equation}
The goal is to minimize this function based on observations of $J(x(t))$,  where the signal $\bfmx$ is chosen by design.  It is assumed that    $J\colon\Re^d\to\Re$ is convex, twice continuously differentiable,  and that it has a unique minimizer, denoted as $\varqsa^*$.  Computation of the optimizer is thus equivalent to obtaining a zero of the gradient of $J$.   The goal is to design QSA algorithms that seek solutions to the equation $\barf(\varqsa^*) =0$, where 
\begin{equation}
\barf(\varqsa) := H \nabla J (\varqsa) \,,\qquad \varqsa\in\Re^d\,, 
\end{equation}
where the matrix $H$ is invertible, and will be part of the algorithm design.

%\rd{We need to discuss to see what we want to say.  I'd like to shorten this to less than 1/2 page.}

Two general algorithms are proposed in the following.  In each case, We design the signal $\bfmx$ as the sum of two terms $x(t) =  \varqsa(t) + \epsy \probe(t)$, $t\ge 0$,  where $\epsy>0$ and 
\begin{equation} \label{eq:probe}
\probe_i(t) = \sqrt{2}\sin (\omega_i t)   
\end{equation}
for $\omega_i \neq \omega_j$ for all $i \neq j$. It can be shown that this process satisfies\footnote{Other perturbation signals can be chosen as well, as soon as they satisfy \eqref{eq:zero_mean} and \eqref{eq:unit_cov}.}:
\begin{eqnarray} 
\lim_{T \rightarrow \infty }\frac{1}{T} \int_{t = 0}^T \probe(t) \, dt
&= & 0
\label{eq:zero_mean}
\\
\lim_{T \rightarrow \infty }\frac{1}{T} \int_{t = 0}^T \probe(t), \probe(t)^\sfT \, dt &=& I
\label{eq:unit_cov}
\end{eqnarray}
where $I$ is the identity matrix.

For a given $\varqsa \in \Re^d$, consider then the second-order Taylor expansion of the objective function around $\varqsa$:
\begin{align*}
&J(\varqsa + \epsy \probe(t)) =
J(\varqsa ) \\
&\quad + 
  \epsy \probe(t)^\sfT \nabla J (\varqsa)
  + \frac{1}{2} \epsy^2 \probe(t)^\sfT \nabla^2 J (\varqsa)  \probe(t)  + o(\epsy^2).
\end{align*}
Define $f (\varqsa, \xi) :=  -\probe J(\varqsa + \epsy \xi)$. It is easy to verify that under  \eqref{eq:zero_mean} and \eqref{eq:unit_cov},
%, and \eqref{eq:asym_ind}, 
one has that: 
\begin{align}
\barf (\varqsa) := \lim_{T \rightarrow \infty }\frac{1}{T} \int_{t = 0}^T f (\varqsa, \xi(t)) \,  dt
=
-\epsy \nabla J (\varqsa) +\text{Err}(\epsy)
\end{align}
where $\| \text{Err}(\epsy)\| \le O(\epsy^2)$.  
Thus, based on \eqref{eq:algo}, the following algorithm seeks for (approximate) zeros of $\nabla J$:

\noindent
\textbf{QSA Gradient Descent \#1: } 
\begin{equation}
\begin{aligned}
\ddt \varqsa(t) &= -  a(t) \probe(t)    J( x(t) )
 \\
 x(t)& =  \varqsa(t) + \epsy \probe(t)\,   .
\end{aligned}
\label{e:ES}
\end{equation}

In the second algorithm, it is assumed moreover that  $\{  \probe(t)\} $  is differentiable with respect to time. We define $\gamma(t) := (\xi(t), \ddt \xi(t)) \in \Re^{2m}$ as the perturbation signal; the reason for this definition will become apparent shortly.

The following limit is assumed to exist,  and the limit is assumed to be invertible: 
\[
 \Sigma_{\bullet} 
 \eqdef  
 \lim_{T\to\infty} \frac{1}{T} \int_0^{T}   \ddt\probe(t) \bigl[ \ddt\probe(t) \bigr]^\sfT  \, dt;
\] 
this is indeed true for the signal chosen in \eqref{eq:probe}.

For a given $\varqsa \in \Re^d$, we have
\[
\ddt J(\varqsa + \epsy \probe(t)) =  \epsy \ddt\probe(t)^\sfT  \nabla J\, (\varqsa + \epsy \probe(t)),   
\]
implying that
\begin{equation}
\begin{aligned}
  \ddt\probe(t)  \ddt J(\varqsa + \epsy \probe(t)) 
  &= 
   \epsy \ddt\probe(t) { \ddt\probe(t)}^\transpose     \nabla J\, (\varqsa + \epsy \probe(t))  
   \\
   & = \epsy \ddt\probe(t) { \ddt\probe(t)}^\transpose \Big [ \nabla J\, (\varqsa) \\
   &\quad \quad + \epsy \nabla^2 J\, (\varqsa) \probe(t) \Big ]  +O(\epsy^2)  
   \\
   &= \epsy \ddt\probe(t) { \ddt\probe(t)}^\transpose  \nabla J\, (\varqsa) + O(\epsy^2). 
\end{aligned}
   \label{eq:approx_f}
\end{equation}
Therefore,
\begin{equation} \label{eq:approx_barf}
\lim_{T \to \infty}\frac{1}{T} \int_0^T \ddt\probe(t)  \ddt J(\varqsa + \epsy \probe(t)) 
=
\epsy \Sigma_{\bullet} \nabla J\, (\varqsa) + O(\epsy^2).
\end{equation}
This motivates 

\noindent\textbf{QSA Gradient Descent \#2:}
\begin{equation}
\begin{aligned}
\ddt \varqsa(t) &= -  a(t) G  \ddt \xi(t)  \ddt J( x(t) )
 \\
 x(t)& =  \varqsa(t) + \epsy \xi(t)\,   ,
\end{aligned}
\label{e:QGD}
\end{equation}
in which $G$ is a given $d\times d$ matrix and is part of the design. The choice $G=\Sigma_{\bullet}^{-1}$ might be used to approximate the steepest descent algorithm.

In view of \eqref{eq:approx_f} and \eqref{eq:approx_barf}, the algorithm \eqref{e:QGD} is approximately equivalent to a QSA algorithm of the form \eqref{eq:algo}, with
\begin{equation*}
f(\varqsa, \gamma) := -\epsy G \gamma_2 \gamma_2^\transpose  \nabla J\, (\varqsa + \epsy \gamma_1), \quad \gamma = (\gamma_1, \gamma_2) \in \Re^{2m}
\end{equation*}
and
\begin{equation*}
\barf (\varqsa) \approx -\epsy G \Sigma_{\bullet} \nabla J\, (\varqsa).
\end{equation*}

\medbreak

%This is a simple instance of the extremum-seeking algorithms of \cite{liukrs12}. 

Either of the two algorithms   can be implemented based on observations of $\{J(  x(t) )\}$, without knowledge of the gradient. In fact, these algorithms are stylized versions of \emph{the extremum-seeking algorithm}  of \cite{liukrs12}.  The gain  $\bfma$ is typically assumed constant in this literature,  and there is a large literature on how to improve the algorithm, such as through the introduction of a linear filter on the measurements $\{ J(x(t))\}$.   It is hoped that the results of this paper can be used to guide algorithm design in this application.

\section{QSA for Reinforcement Learning}  %: Policy Evaluation} 
\label{sec:RL}

\def\dyn{g}
 
In this section we will show how QSA can be used to speed up the exploration phase which is needed for policy evaluation in reinforcement learning. 

\subsection{Off-policy TD Learning}
Consider the nonlinear state space model 
\[
\ddt x(t) = \dyn(x(t),u(t)) \,, \qquad t\ge 0\,  
\] 
with $x(t)\in\Re^n$,  $u(t)\in\Re^m$.   Given a cost function $c\colon\Re^{n+m}\to \Re$, and a feedback law $u(t) = \fb(x(t))$,  let $J$ denote the associated value function:
%%%%%%%%%%%%%%%%%%%%%%%%%%%%%%%%%%% 
\[
J^\fb(x) =  \int_0^\infty c(x(t),\fb(x(t)))\, dt\,, \qquad x=x(0). 
\]
The goal of policy evaluation (or TD-learning~\cite{tsiroy97a}) is to approximate this value function based on input-output measurements.   
%\textcolor{blue}{In particular, it is known that the TD(1) algorithm will minimize the mean square error of value functions $\|J^\phi - J^\varqsa\|^2$ for a particular norm (in a discrete-time, Markovian setting).  \textbf{Marcello: I don't understand this sentence, $J^\varqsa$ undefined}}
It is assumed in \cite{tsiroy97a} that the joint process $(\bfmx,\bfmu)$ is an ergodic Markov chain, which presents an obvious challenge in this deterministic setting:     this ergodic steady state will typically be degenerate.  It is common to introduce noise, as in Q-learning~ \cite{mehmey09a},  and also a discount factor in the definition of $J$ to ensure that  $J(x)<\infty $   for all $x$.    Following these modifications,   the  approximation objective has been changed significantly:  rather than approximating the original value function $J$,  the algorithm will provide an approximation for the value function with discounting, and with a randomized policy.   
Sufficient exploration and/or discounting may create significant distortion in the value function. 

The algorithm proposed here avoids these difficulties.   The construction begins with a Q-function~\cite{mehmey09a}  defined with respect to the given policy:
\[
Q^\fb(x,u) = J^\fb(x) + c(x, u) +\dyn(x,u )\cdot \nabla J^{\fb}\, (x).
\]
%This is one of two steps in the policy improvement algorithm (PIA):  the second step is to obtain a new policy:
%\begin{equation}\label{eq.policy.update}
%\fb^+(x) =\argmin_u Q^\fb(x,u)
%\end{equation}
%Our goal is to approximate $Q^\fb$ so that we can approximate the PIA algorithm. 
This function solves the fixed point equation
\begin{align}\label{eq:Q:pi}
Q^\fb(x,u) = \uQ^\fb(x) + c(x, u) +\dyn(x,u )\cdot \nabla \uQ^\fb\, (x)
\end{align}
in which we use the notational convention $\uF(x) = F(x,\fb(x))$ for any function $F$. We consider a family of functions $Q^{\fb\theta} (x,u)$ parameterized by $\theta$, and define the Bellman error for a given parameter as
\begin{align}
\begin{split}
\clE^\varqsa(x,u) &= -Q^{\fb\varqsa} (x,u)+ \uQ^{\fb\varqsa} (x) + c(x,u) \\
& +  \dyn(x,u)\cdot \nabla \uQ^{\fb\varqsa}\, (x)  
\end{split}
\label{e:BEQ}
\end{align}

The goal of policy evaluation is to create a data-driven algorithm that, without using information on the system's model, computes a parameter $\varqsa^*$ for which the Bellman error is small:  for example,   minimizes $\|\clE^\varqsa\|$ in a given norm.  In \cite{mehmey09a}, ideas from~\cite{farroy03a} are used  to construct a convex program for a related learning objective. In this paper, we propose an \emph{off-policy} RL algorithm: the value function for $\fb$ is approximated while the actual input $u$ of the system may be entirely unrelated. 

We choose a feedback law with ``excitation'', of the form
\begin{equation}
u(t) = \kappa(x(t),\xi(t))
\label{s:QSAfeedback}
\end{equation}
where $\kappa$ and $\bfxi$ are such that the resulting state trajectories are bounded for each initial condition, and that the joint process $(\bfmx,\bfmu,\bfxi)$ admits an ergodic steady state.
%(further assumptions can be found in the paper \andrey{Which paper? \cite{mehmey09a}?}). 
The goal is to find $\varqsa^*$ that minimizes the mean square error:
\begin{equation}
\|\clE^\varqsa\|^2
\eqdef
\lim_{T\to\infty}
\frac{1}{T} \int_0^T  \bigl[
\clE^\varqsa(x(t),u(t))  \bigr]^2\, dt.
\label{e:ergodicBellman}
\end{equation}

Similarly to Section \ref{sec:grad_free}, the first-order condition for optimality is expressed as  a root-finding problem:  
$\nabla_\varqsa\|\clE^\varqsa\|^2 = 0$.   
%\andrey{This is under assumption that the objective function is convex, differentiable with unique minimum.}
Collecting together the definitions, we arrive at the following 
QSA steepest descent algorithm: 
\begin{equation}
\begin{aligned}
\ddt \varqsa (t) &= - a(t) \clE^{\varqsa(t)} (x(t),u(t)) \zeta^{\varqsa(t)}(t) 
\\
\zeta^\varqsa(t) & := \nabla_{\varqsa}\clE^{\varqsa} (x(t),u(t)) 
\end{aligned}
\label{e:Q}
\end{equation}
The vector process $\{\zeta^{\varqsa(t)}(t)\}$ is analogous to the \textit{eligibility vector} defined in TD-learning \cite{sze10,sutbar98,bertsi96a}.

\textit{Model-free realization.}   It appears from the definition \eqref{e:BEQ} that the nonlinear model must be known.   A model-free implementation is obtained on recognizing that for any parameter $\varqsa$,  and any state-input pair $(x(t),u(t))$, 
\begin{equation}
\begin{aligned}
\clE^\varqsa(x(t),u(t)) & = -Q^{\fb\theta} (x(t),u(t))+ \uQ^{\fb\theta} (x(t)) 
\\
&\qquad + c(x(t),u(t))  +  \ddt \uQ^{\fb\theta} (x(t))
\end{aligned} 
\label{e:BEonline}
\end{equation} 

\noindent
\textit{(Approximate) Policy improvement algorithm (PIA):}  Given a policy $\fb$ and  approximation   $Q^{\fb\theta^*}$, the policy is updated:
\begin{equation}\label{eq.policy.update}
\fb^+(x) =\argmin_u Q^{\fb\theta^*}(x,u)
\end{equation}
This procedure is repeated to obtain a recursive algorithm.

% and repeat iteratively policy evaluation~\eqref{e:Q} and policy improvement~\eqref{eq.policy.update}.

\subsection{Practical Implementation}
Given a basis of functions $\{\psi_i : 1\le i\le d\}$,  consider the linearly parameterized family
\begin{equation}
Q^{\fb\theta} (x,u) =  d(x,u) + \theta^\transpose \psi(x,u)\,,\quad \theta\in\Re^d. 
\label{e:Qtheta}
\end{equation}
%Given a parametrized family of functions $Q^{\fb,\varqsa}$,   the Bellman error as a function of $(x(t),u(t)) $ is now defined by
%\begin{equation}
%\begin{aligned}
%\clE^{\fb,\varqsa}(x(t),u(t))  &= -Q^{\fb,\varqsa}(x(t),u(t)) + \uQ^{\fb,\varqsa} (x(t)) 
%\\
%&\qquad + c(x(t),u(t))  +   \ddt   \uQ^\fb\, (x(t))
%\end{aligned}
%\label{e:BEQfixed}
%\end{equation}
% 
Note that the Bellman error is a linear function of $\varqsa$ whenever this is true of $Q^{\fb,\varqsa}$.    
Consequently, minimization of \eqref{e:ergodicBellman} is a model-free linear regression problem,  and the limit exists for any stable input. Moreover, the steepest descent algorithm~\eqref{e:Q} becomes linear. In fact, given~\eqref{e:Qtheta}, we define
\begin{align*}
\zeta(t)  &:= \left[\psi(x(t),\fb(x(t))) - \psi(x(t),u(t)) +\right.\\
&\left.\quad \ddt \psi(x(t),\fb(x(t))) \right]\\
b(t)  &:= \left[ c(x(t), u(t)) - d(x(t), u(t)) + d(x(t), \fb(x(t)))\right. \\
&\left.\quad + \ddt d(x(t), \fb(x(t))) \right] 
\end{align*}
Then $\clE^{\fb,\varqsa}(x(t),u(t)) = b(t)+\zeta(t)^\top \theta$, and~\eqref{e:Q} becomes
\begin{equation}
\begin{aligned}
\ddt \varqsa (t) &= - a(t)\left[\zeta(t)^\top \,\varqsa (t) + b(t)  \right]\zeta(t)\\
\end{aligned}
\label{e:Q:fixed:policy}
\end{equation}
The convergence of~\eqref{e:Q:fixed:policy} may be very slow if the matrix 
\begin{equation}
G := \lim_{t\to\infty} \frac{1}{t}\int_{0}^t \zeta(\tau)\zeta(\tau)^\top \mathrm d \tau
\label{e:G}
\end{equation}
is poorly conditioned (i.e., has some eigenvalues close to zero). Note that using $G^{-1}$ as a matrix gain could solve this problem. The integral~\eqref{e:G} can be estimated from data. This suggests an intuitive two step procedure for the steepest descent algorithm~\eqref{e:Q:fixed:policy}
\begin{subequations}\label{e:fixed:q:matrix}
\begin{align} 
\haG_t &= \frac{1}{t}\int_{0}^t \zeta(\tau)\varrho(\tau)^\top \mathrm d \tau,\quad  0\le t\le T \\
\ddt \varqsa (t) &= - a(t)\haG_T^{-1} \left[\zeta(t)^\top \,\varqsa (t) + b(t)  \right]\zeta(t),\,  t\ge T 
\label{e:Q:fixed:policy:matrix}
\end{align}
\end{subequations}
The results in Section~\ref{s:QSA} suggest that this is indeed a good idea in order to achieve the optimal convergence rate   $\mathcal O (1/t) $.  

\subsection{Numerical example} 
% \textcolor{blue}{Marcello: The way I actually had implemented is \textbf{with} the quadratic part of $u$, this is because then we do not need to know $c(x,u)$ either but we read it as an incurred cost. I understand that this way I cannot guarantee that the minimization step is well defined. For the time being I write it this way and I will change the code (or the text) later }
Consider the LQR problem in which $\dyn(x,u) = Ax+Bu$,  and $c(x,u) = x^\top Mx + u^\top Ru$, with $(A,B)$ controllable, $M\ge 0$ and $R>0$.   Given the known structure of the problem, we know that the function $Q^\phi$ associated with any linear policy $\phi(x) = Kx$, takes the form
\[
Q^\phi = \begin{bmatrix} x \\ u  \end{bmatrix}^\top \left(
\begin{bmatrix} M & 0\\0 & R  \end{bmatrix}
 + 
 \begin{bmatrix} A^\top P + P A  + P   & PB \\
 B^\top P & 0  \end{bmatrix}
\right)
\begin{bmatrix} x \\ u  \end{bmatrix},
\]
where $P$ solves the Lyapunov equation $A^\top P + PA + K^\top R K + Q = 0$ and therefore lies within the parametric class \eqref{e:Qtheta} in which $d(x,u) = c(x,u)$ and each $\psi_i$ is a quadratic function of $(x,u)$. For example,  for the special case $n=2$ and $m=1$, we can take the quadratic basis
\[
\{\psi_1,\dots,\psi_6\}
=
\{
x_1^2,x_2^2, x_1 x_2, x_1 u, x_2 u, u^2\}
\]

%We could omit the $u^2$ term in the parametrization. When we do so, however, we observe numerical instabilities in the algorithm (the explanation is not yet clear). 
In order to implement the  algorithm~\eqref{e:Q:fixed:policy:matrix} we begin with selecting an input of the form
\begin{align}\label{e:offpolicy}
u(t) = K_0 x(t) + \xi(t)
\end{align}
where $K_0$ is a stabilizing controller and $\xi(t) =\sum_{j=1}^q a_j \sin(\omega_j t + \phi_j)$. Note that $K_0$ need not be the same $K$ whose value function we are trying to evaluate. 

\spm{Say at the start that we are only looking for  a stationary point in general, but here we end up with a simple quadratic objective, no?
\\
Then we run~\eqref{e:Q:fixed:policy:matrix} to obtain $\varqsa^\star$ which minimizes the Bellman error~\eqref{e:ergodicBellman}. 
}
\begin{figure}[thb]
\begin{center}
\input{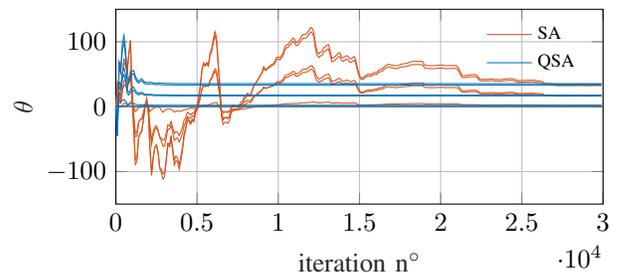}
\caption{\small    Comparison of QSA and Stochastic Approximation (SA) for policy evaluation. It is observed that QSA converges significantly faster.
%This is in line with the theoretical results presented in Section~\ref{s:QSA}. 
}
\label{fig.qsa.PI}
\end{center}
\vspace{-1.5em}
\end{figure}

The algorithm was tested on the simple LQR example where the system is a double integrator with friction
\begin{equation}\label{e:lin}
\dot x = \begin{bmatrix} 0 & -1\\0 & -0.1 \end{bmatrix} x + \begin{bmatrix}0\\1\end{bmatrix} u, \quad M = I, \quad R = 10\, I .
\end{equation}
Figure~\ref{fig.qsa.PI} shows the evolution of the QSA algorithm for the evaluation of the policy $K=[-1,0]$ using the stabilizing controller $K_0 = [-1,-2]$ and $\xi$ in~\eqref{e:offpolicy} as the sum of 24 sinusoids with random phase shifts and whose frequency was sampled uniformly between $0$ and $50$ rad/s. The QSA algorithm is compared with the related SA algorithm in which $\xi$ is ``white noise'' instead of a deterministic signal (formalized as an SDE). For implementation, both~\eqref{e:fixed:q:matrix} and the linear system~\eqref{e:lin} were discretized with forward Euler discretization; time-step of 0.01s.

\begin{figure}[thb]
\begin{center}
\definecolor{mycolor1}{rgb}{0.00000,0.44700,0.74100}%
\begin{tikzpicture}

\begin{axis}[%
width=0.8\columnwidth,
height=0.25\columnwidth,
at={(0,0)},
scale only axis,
xmin=1,
xmax=6,
xlabel style={font=\color{white!15!black}},
xlabel={Policy improvement $n^{\circ}$},
ymin=0,
ymax=11,
ylabel style={font=\color{white!15!black}},
ylabel={$\|K-K^{\star}\| / \|K^{\star}\|$},
axis background/.style={fill=white},
xmajorgrids,
ymajorgrids
]
\addplot [color=mycolor1, 
line width=2pt,   
mark size=2.5pt,
mark=x,
mark options={solid},
forget plot]
  table[row sep=crcr]{%
1	10.4796828118461\\
2	4.82249144332938\\
3	2.02892096522203\\
4	0.700134726208808\\
5	0.151171489672945\\
6	0.00974789234373565\\
};
\end{axis}

\end{tikzpicture}%
\caption{\small Iterations of the policy improvement algorithm (PIA)~\eqref{eq.policy.update} where each evaluation is performed by the model-free algorithm~\eqref{e:fixed:q:matrix}. We observe that the PIA algorithm indeed converges to the optimal controller $K^\star$.}
\vspace{-1.5em}
\label{fig.PIA}
\end{center}
\end{figure}
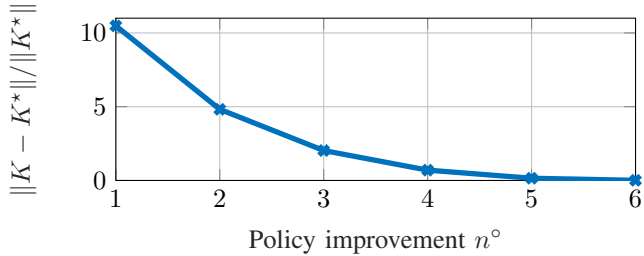
In Figure~\ref{fig.PIA} we show the distance of the iterates of the policy improvement algorithm~\eqref{eq.policy.update} and the optimal controller $K^\star$ (which can be easily computed for an LQR problem). Each policy evaluation performed by the model-free algorithm~\eqref{e:fixed:q:matrix}.

\section{Convergence Analysis}
\label{s:QSA}

The extension of stability and convergence results from the classical stochastic model 
\eqref{eq:SA} to the deterministic analog \eqref{eq:algo} requires some specialized analysis since the standard methods are not directly applicable.  In particular,  the first step in \cite{bor08a} and other references is to 
write \eqref{eq:SA} in the form,
\[
\varqsa_{n+1} = \varqsa_n + a_n \bigl(\barh(\varqsa_n) + M_n\bigr) \, ,
\]
where $\bfmM$ is a martingale difference sequence (or a perturbation of such a sequence).  This is possible when $\bfmW$ is i.i.d.,  or for certain Markov $\bfmW$. %  \cite{mamakshw90}.    
A similar transformation is not possible for any class of deterministic $\bfxi$.

%%%%%%%%%%%%%%%%%%%% 

\subsection{Assumptions for convergence}
\label{ss:algorithmQSA}

% In the ODE approach for analysis of stochastic approximation algorithms,   a continuous time trajectory $\bfvarqsa$ is obtained from the sequence defined in \eqref{eq:SA} through linear interpolation:
% The $n^{th}$ iterate $\varqsa_n$ in the recursion corresponds to a point $\varqsa(u_n)$ in the interpolated trajectory, where
% \[
% 	u_n = \sum_{i=0}^n a_i.
% \]
% The piecewise linear trajectory  $\bfvarqsa$   is then compared with $\bfvarode$ that is defined in \eqref{eq:avgODE}.  

As in standard analysis of SA, the starting point is a temporal transformation:
substitute in \eqref{eq:algo} the new time variable given by
\[
	u= g(t) \eqdef  \int_0^t a(r)\, dr,\qquad t\ge 0.
\] 
The time-scaled process is then defined by  
\begin{equation}
\varscaled(u) \eqdef  \varqsa(g^{-1}(u)).
\label{e:varscaled}
\end{equation}  
For example, if $a(r)=(1+r)^{-1}$, then 
\begin{equation}
u=\log(1+t) \ \
	\text{and} \ \ 
\xi(g^{-1}(u))) = \xi(e^u-1).
\label{e:a-r-inv}
\end{equation}

The chain rule of differentiation gives 
\begin{equation*}
\ddu\varqsa(g^{-1}(u)) = f(\varqsa(g^{-1}(u)),\xi(g^{-1}(u))).
\end{equation*}  
That is,  the time-scaled process solves the ODE,
\begin{equation}
\label{eq:algoscaled}
\ddu\varscaled(u) = f(\varscaled(u),\xi(g^{-1}(u))).
\end{equation}
The two processes $\bfvarqsa$ and $\bfvarscaled$ differ only in time scale, and hence, proving convergence of one  proves that of the other. For the remainder of this section we will deal exclusively with $\bfvarscaled$;  it is on the `right' time scale for comparison with  $\bfvarode$, the solution of \eqref{eq:avgODE}.

\noindent
\textbf{Assumptions:}
\begin{romannum} %so we can control spacing

\item[(A1)] The system described by equation \eqref{eq:avgODE} has a globally asymptotically stable equilibrium at $\varqsa^*$.

\item[(A2)] There exists a continuous function $V:\Re^d\rightarrow\Re_+$ and a constant $c_0>0$ such that, for any initial condition $\varode(0)$ of
\eqref{eq:avgODE}, and
any $0\le T\le1$,
the following bounds hold whenever
$\|\varode(s)\|>c_0$,
\[
V(\varode(s+T)) - V(\varode(s)) \le -T\|\varode(s)\|.
\]

%\andy{Remove "there exists"...}

% There exists a function $f:\Re^d\times\Re^m\rightarrow\Re^d$ and a process $\{\xi(t)\}_{t\ge0}$ that takes values in a compact set $\Omega\subset\Re^m$ 

\item[(A3)]  There exists a constant $b_0<\infty$, such that for all $\varqsa\in\Re^d$, $T>0$,
%\notes{$b_0$ first to emphasize it is independent of theta}
\[ 
\left\|\frac{1}{T}\int_0^T f(\varqsa,\xi(t))\, dt - \bar{f}(\varqsa)\right\| \le \frac{b_0}{T}(1+\|\varqsa\|) 
\]

\item[(A4)] There exists a constant $L <\infty$ such that the functions $V$, $\bar{f}$ and $f$ satisfy the following Lipschitz conditions:
\begin{align*}
\|V(\varqsa') - V(\varqsa)\| &\le L \|\varqsa' - \varqsa\|, 
\\
\|\bar{f}(\varqsa') - \bar{f}(\varqsa)\| &\le L\|\varqsa' - \varqsa\|, 
\\
\|f(\varqsa',\xi) - f(\varqsa,\xi)\| &\le L\|\varqsa' - \varqsa\|,\quad \varqsa', \, \varqsa\in\Re^d\,, \ \xi\in\Re^m
\end{align*}
 
\item[(A5)] The process $\bfma$ is non-negative and monotonically decreasing, and as $t\rightarrow\infty$, 
\[
a(t)\downarrow 0,\qquad\int_0^t a(r)\, dr \rightarrow \infty.
\]
\end{romannum}
Assumption (A1) determines uniquely  the possible limit point of the algorithm.  Assumption (A2) ensures that there is a Lyapunov function $V$ with a strictly negative drift whenever $\bfvarode$ escapes a ball of radius $c_0$. This assumption is used to establish boundedness of the trajectory $\bfvarscaled$. Assumptions (A3) and (A4) are technical requirements essential to the proofs:
 (A3) is only slightly stronger than ergodicity of $\bfxi$ as given by \eqref{eq:ergodic}, while (A4) is necessary to control the growth of the respective functions. The process $\bfma$ in (A5) is a continuous time counterpart of the standard step size schedules in stochastic approximation, except that we impose monotonicity in place of square integrability.

\noindent \emph{Verifying (A2) for a linear system.}  
Consider the ODE \eqref{eq:avgODE} in which $\barf(x) = Ax$ with $A$ a Hurwitz $d\times d$ matrix.   There is a quadratic function $V_2(x) =   x^\transpose P x$  satisfying the Lyapunov equation $PA +A^\transpose P = -I$, with $P>0$.  
% Consequently, solutions to \eqref{eq:avgODE}  satisfy
% \[
% \ddt V_2(\varode(t))  = -\| \varode(t)\|^2
% \]
The function $V=k\sqrt{V_2}$, where the constant $k>0$ is chosen suitably large,
% \[
% \ddt V(\varode(t))  =     
% - \frac{k}{2}  \frac{1}{\sqrt{V_2(\varode(t)) } } \| \varode(t)\|^2   
% \]
 is a Lipschitz solution to (A2) for some finite $c_0$.  
% For $k>0$ sufficiently large we obtain
% \[
% \ddt V(\varode(t))  
%          \le  -\| \varode(t)\|
% \]
% and hence this $V$ is a Lipschitz solution to (A2).
% , for  $c_0>0$ sufficiently large.   \spm{I need to give this a tweak}
\qed 
 
 \subsection{Convergence}

%The QSA algorithm is consistent under these assumptions:

The following is our main convergence result. 

\begin{theorem}
\label{t:convergence} 
Under Assumptions (A1)--(A5),  the solution to \eqref{eq:algo} converges to $\varqsa^*$ for each initial condition. 
\end{theorem}

% Assumption (A1) determines uniquely  the possible limit point of the algorithm.  Assumption (A2) ensures that there is a Lyapunov function $V$ with a strictly negative drift whenever $\bfvarode$ escapes a ball of radius $c_0$. This assumption is used to establish boundedness of the trajectory $\bfvarscaled$. Assumptions (A3) and (A4) are technical requirements essential to the proofs:
% (A3) is only slightly stronger than ergodicity of $\bfxi$ as given by \eqref{eq:ergodic}, while (A4) is necessary to control the growth of the respective functions. The process $\bfma$ in (A5) is a continuous time counterpart of the standard step size schedules in stochastic approximation, except that we impose monotonicity in place of square integrability.

% The extremum seeking algorithm \eqref{e:ES} satisfies (A4)  if $J$ is Lipschitz continuous.

% Boundedness of the trajectory $\bfvarscaled$ for each initial condition is the first step in the proof of convergence.  It is lengthy, but follows along arguments similar to those used in the classical SA literature.  Details can be found in   \cite{berchecoldalmeymeh19}.   

Define $\varode^u(w)$, $w\ge u$, to be the unique solution to \eqref{eq:avgODE} `starting' at $\varscaled(u)$:
\begin{equation}
\label{eq:init}
\ddw\varode^u(w) = \bar{f}(\varode^u(w)),\, \, w\ge u,  \, \, \varode^u(u) = \varscaled(u).
\end{equation}
The following result is required to prove Theorem \ref{t:convergence}.

 \begin{lemma}
\label{thm:limsuplimsup}
Under the assumptions of  \Theorem{t:convergence}, 
for any $T>0$,  as $t\to\infty$,
\[  \sup_{v\in[0,T]}\Bigl\|
		\int_u^{u+v} \bigl[ f(\varscaled(w),\xi(g^{-1}(w))) - \bar{f}(\varscaled(w) ) \bigr]\, dw
							\Bigr \| \to 0
\]
and as $u\to \infty$,  
$
 \sup_{v\in[0,T]}\|\varscaled(u+v) - \varode^u(u+v)\| \to 0. 
$ \qed
\end{lemma}
The proof of \Lemma{thm:limsuplimsup}
is contained in the Appendix; 
the second limit is similar to Lemma~1 in Chapter~2 of \cite{bor08a}.

\subsubsection*{Proof of \Theorem{t:convergence}}

The first step in the proof is to establish ultimate boundedness of $\varscaled(u)$:
there exists $b < \infty$ such that for each $\varqsa\in\Re^d$,  there is a $T_\varqsa$ such that
\[
	\|\varscaled(u)\|\le b \textrm{\ \ for all } u\ge T_\varqsa\,,
    \  \varscaled(0) = \varqsa
\]
The (lengthy) proof is contained in the Appendix; see Proposition \ref{t:stability} there. 

Thus, for $u\ge T_\varqsa$, $\|\varode^u(u)\| = \|\varscaled(u)\|\le b$. By the definition of global asymptotic convergence, for every $\epsy>0$, there exists a $\tau_\epsy>0$, independent of the value $\varode^u(u)$, such that 
$
	\|\varode^u(u+v) - \varqsa^*\| < \epsy \textrm{\ \ for all }v\ge \tau_\epsy
	$.
  \Lemma{thm:limsuplimsup} gives,
\begin{align*}
	\limsup_{u\rightarrow\infty}\|\varscaled(u+&\tau_\epsy) - \varqsa^*\| 
	\\
	&\le \limsup_{u\rightarrow\infty}\|\varscaled(u+\tau_\epsy)-\varode^u(u+\tau_\epsy)\| 
	\\
	&\qquad + \limsup_{u\rightarrow\infty}\|\varode^u(u+\tau_\epsy) - \varqsa^*\| 
 \le \epsy. 
\end{align*}
Since $\epsy$ is arbitrary, we have the desired limit.
% \[
% 	\lim_{u\rightarrow\infty} \|\varscaled(u) - \varqsa^*\|  =\lim_{u\rightarrow\infty} \|\varscaled(u+\tau_\epsy) - \varqsa^*\|  = 0.
% \] 
\qed

%------------------------------------------------------------------

\subsection{Variance} 
 
Let $\tilvarqsa (t) \eqdef \varqsa(t) - \varqsa^* $ and $\nu(t)=(t+1)\tilvarqsa(t)$. 
This section is devoted to  providing conditions under which $\bfnu$ is bounded, and there is a well defined covariance: 
\begin{equation}
\barSigma_\varqsa\eqdef \lim_{T\to\infty}  \frac{1}{T} \int_0^T  \nu(t)\nu(t)^\transpose \, dt.  
\label{e:cov}
\end{equation}
Analysis requires  additional assumptions on the ``noise'' process.   It is also assumed that the model is linear and stable:   
\begin{romannum}
\item[(A6)]  
The function $f$ is linear,  $f(\varqsa,\xi) = A\varqsa + \xi$, the gain is  $a(t)=1/(t+1)$, and   
\begin{romannum}
\item
$A$ is Hurwitz, and each eigenvalue $\lambda(A)$ satisfies   $
\Real(\lambda) < -1$.
  
\item  The function of time $\bfxi$ is bounded, along with its partial integrals, denoted
\[
\begin{aligned}
\xi^I(t)& = \int_0^t \xi(r)\, dr,  \qquad \xi^{I\!I}(t) = \int_0^t \xi^I(r)\, dr.
% \\
% \xi^{I\!I}(t)& = \int_0^t \xi^I(r)\, dr.
\end{aligned}
\]
\end{romannum}
\end{romannum} 

Assumption (A6) implies that $\barf(\varqsa)= A\varqsa$, so that $\varqsa^*=0$.
The linearity assumption is typical in much of the literature on variance for stochastic approximation \cite{kontsi04,kusyin97,bor08a}.  As in the SA literature, it is likely that the results of this section can be extended to nonlinear models via a Taylor-series approximation.
 
A typical example of Assumption (A6ii) is the case where  the entries of $\bfxi$ can be expressed as a sum of sinusoids: 
\begin{equation} 
\xi (t) = \sum_{i=1}^K  v^i   \sin ( \phi_i + \omega_i t  )
\label{eq:InputQSA}
\end{equation}  
for fixed vectors $\{v^i\}$, phases $\{\phi_i\}$,  and  frequencies   $\{\omega_i\}$.   

  \Theorem{t:var} below implies that $\| \nu(t) - \xi^I(t)\|\to 0$,
as $t\to\infty$.  Consequently, the error covariance   exists whenever there is a covariance for $\bfxi^I$: 
\[
\barSigma_\varqsa 
=
\Sigma_{\xi^I} \eqdef  \lim_{T\to\infty}  \frac{1}{T} \int_0^T  \xi^I(t){\xi^I(t)}^\transpose \, dt.  
\]
This is easily computed for the special case \eqref{eq:InputQSA}.

Let $\barA \eqdef I+A$ and 
fix a constant $\epsy_S$  satisfying $0<\epsy_S< -\Real(\barlambda)$ for each eigenvalue $\barlambda$ of $\barA$; this is possible due to Assumption (A6i).  
Associated with the ODE $\ddt x(t) = (1+t)^{-1} \barA x(t)$ is the \textit{state transition matrix}:
 \begin{equation}
S(t;r) = \exp\Bigl(\log\Bigl[ \frac{1+t}{1+r} \Bigr ]\barA\Bigl) \,,\quad r,t\ge 0.
\label{e:stateTrans}
\end{equation}  
It is easily shown that it satisfies the defining properties
\begin{equation}
S(t;t)=I\,,\quad  \ddt S(t;r) = \frac{1}{t+1} \barA  S(t;r)\,,\quad r,t\ge 0.
\label{e:STM}
\end{equation}

\begin{theorem}
\label{t:var}
Suppose Assumptions (A1)--(A6) hold.  Then,  
for each initial condition $\varqsa(0)$, 
\begin{equation}
\tilvarqsa(t)  =   \frac{1}{t+1} \Bigl[ \xi^I(t) 
+S(t;0) \tilvarqsa(0)\Bigr]
	+O\Bigl( \frac{1}{(t+1)^{1+\delta_S}}\Bigr), 
    \label{e:varSimple}
\end{equation}
where $\delta_S=\min(\epsy_S,1)$,  and the final error term is independent of the initial condition $\tilvarqsa(0)$. 
%and 
%\begin{equation}
%S(t;r) = \exp\Bigl(\log\Bigl[ \frac{1+t}{1+r} \Bigr ]\barA\Bigl) \,,\quad r,t\ge 0.
%\label{e:stateTrans}
%\end{equation} 
Consequently,  the scaled error process satisfies the bound
\begin{equation}
    \nu(t)   =  \xi^I(t)  +O\Bigl(  \frac{1+\|\tilvarqsa(0)\|}{(t+1)^{\delta_S}}\Bigr). 
% \begin{aligned}
% % \varqsa(t) & = \varqsa^*
% % +  \frac{1}{t+1}  \xi^I(t) 
% % 	+O\Bigl(  \frac{1}{(t+1)^{1+\delta_S}}\Bigr) 
% %     \\[.5em]
% \end{aligned}
\label{e:BestErrorEvolutionFormula}
\end{equation}   
\qed
\end{theorem}
%The matrix-valued function $S(t;r)$ in  \eqref{e:stateTrans} is   the \textit{state transition matrix} for the ODE 
%\[ 
%\ddt x(t) = \frac{1}{t+1} \barA x(t). 
%\]
%It is easily shown that it satisfies the defining properties
%\begin{equation}
%S(t;t)=I\,,\quad  \ddt S(t;r) = \frac{1}{t+1} \barA  %S(t;r)\,,\quad r,t\ge 0.
%\label{e:STM}
%\end{equation}

% \Lemma{t:STM} below implies that the term $ (t+1)^{-1}S(t;0) \tilvarqsa(0)$ appearing in 
% \eqref{e:varSimple}
% can be absorbed into the error term.  It is included in \eqref{e:varSimple} only to emphasize dependence on the initial condition.  That is,  \eqref{e:varSimple} combined with \Lemma{t:STM}  give

The remarkable coupling bound 
\eqref{e:BestErrorEvolutionFormula} follows from \eqref{e:varSimple} and
\Lemma{t:STM} below.  Coupling is illustrated here using the simple integration experiment of
\Section{s:QMC}. The representation \eqref{e:QMC} must be modified to fit the assumptions of the theorem.   First,  denote by $\bfxi^0$  a periodic function of time whose sample paths define the uniform distribution on $[0,1]$: 
for any continuous function $c$,
\[
\lim_{T\to\infty}\frac{1}{T} \int_0^T c(\xi^0(t)) \, dt   = \int_0^1 c(x)\, dx.
\] 
% for example, the sawtooth  function,  $\xi^0(t) = t$  (mod 1).
Introduce a gain $g>0$, and consider the error equation,
% $\bftilvartheta $ for the resulting algorithm:
\begin{equation}
 \ddt \tilvarqsa(t) =  \frac{g}{t+1} [y(\xi^0(t)) - \varqsa^* - \tilvarqsa(t) ]
 \label{e:QMCg}
\end{equation}
The assumptions of the theorem are satisfied with $A=-g$ and $\xi(t)= g [y(\xi^0(t)) - \varqsa^*]$.

\spm{This is an honest description of what I simulated.
I didn't get such nice results with gain g/(g+t).  Remember, in the Euler approximation the real gain is dt *g/(1+t),   so the large g isn't such a big deal.   
}
Figures~\ref{f:plotsQMC}  and \ref{f:hists}  
illustrate the qualitative conclusion of \Theorem{t:var}:  that it is useful to choose $g>1$ in  \eqref{e:QMCg}, so that Assumption (A6i) is satisfied. %In these experiments,  $\bfxi^0$ was chosen to be the sawtooth  function,  $\xi^0(t) = t$  (mod 1).   

Coupling is illustrated in \Fig{f:nu}.  The scaled errors  $g^{-1}\bfnu$ are compared since $\bfxi$ grows linearly with $g$:  we expect  $ g^{-1}\nu(t)\approx \int_0^t (y(\xi^0(r)) - \varqsa^*) $ for large $t$.
The initial condition was set to $\varqsa(0)=10$  in each experiment.   

\begin{figure}[ht]
\Ebox{1}{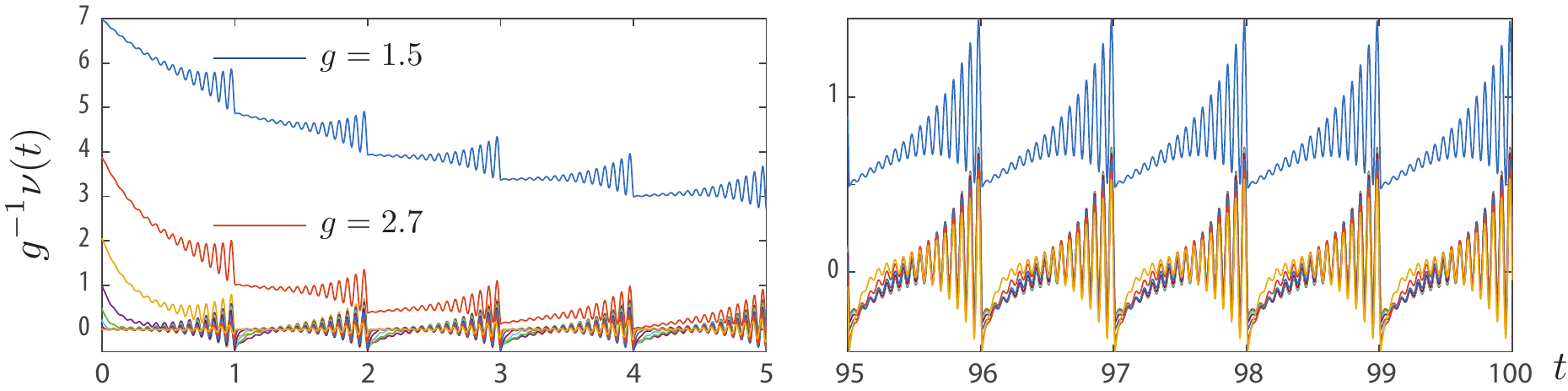} 
\caption{\small  Evolution of $\nu(t)=(1+t)\tilvarqsa(t)$ using  Quasi Monte-Carlo estimates for a range of gains.} 
\vspace{-1em}
\label{f:nu}
\end{figure}

The figure shows results using ten gains,  approximately equally spaced on a logarithmic scale.  The smallest gain is $g=1.5$,  and all other gains satisfy $g\ge 2$.   \Theorem{t:var} asserts that
$|\nu(t) -\xi^I(t) | = O\bigl([1+t]^{-\delta_S}\bigr)$, where $\delta_S<0.5$ for $g=1.5$,  and $\delta_S=1$ for $g\ge 2$. 
  The scaled errors   $\{g^{-1}\nu(t) : 95\le t\le 100\}$ are nearly indistinguishable when $g\ge 2$.  The slower convergence for $g=2.7$ is probably due to the term $S(t;0) \tilvarqsa(0)$ appearing in \eqref{e:varSimple}. 

% The slower rate of convergence is clearly observed for the value $g=1.5$.

Results using gains $g\le 1$ are omitted.  
As expected, $\bfnu$ is unbounded for $g<1$.  
For $g=1$ the approximation \eqref{e:BestErrorEvolutionFormula} fails since 
$\nu(t)$ evolves near $\nu(0)$ for the entire run.

\smallskip

The proof of Theorem \ref{t:var} leverages the following auxiliary results.  Let 
$\tilnu(t)=\nu(t)-\xi^I(t)$,
$t\ge 0$,
denote the ``second-order'' error process.

\begin{lemma}
\label{t:STM} 
The scaled error processes 
solve the respective linear differential equations
 \begin{equation}
 \begin{aligned}
 \ddt \nu(t) &= \frac{1}{t+1} \barA \nu(t) + \xi(t) 
 \\
\ddt \tilnu(t) & = \frac{1}{t+1} \barA \tilnu(t) + \frac{1}{t+1} \barA \xi^I(t)
% \,,\  \nu(0)=\tilnu(0)=\tilvarqsa(0)
% \, . 
 \end{aligned}
\label{e:tilnuDyn}
\end{equation}
The ODE for the second-order error admits the solution
 \begin{equation}
  \tilnu(t) 
   =   S(t,0)\tilvarqsa(0) 
   + \int_0^t  \frac{1}{r+1}  S(t;r)  \barA \xi^I(r)\, dr % \,,\qquad t\ge 0  
  \label{e:tilnuDynSolved}
\end{equation}
where $S$ is defined in \eqref{e:stateTrans}.
Under the eigenvalue assumptions in (A6), there exists  $b_S<\infty$ such that
\[
\| S(t;r) \|_2 \le b_S\Bigl[ \frac{1+t}{1+r} \Bigr ]^{-\epsy_S}
\]
where $\| S(t;r) \|_2$ denotes the  maximal singular value. 
%induced operator norm on L2:  
%https://en.wikipedia.org/wiki/Matrix_norm#Matrix_norms_induced_by_vector_norms
\end{lemma}

\begin{proof}
The representation follows from the state transition matrix interpretation \eqref{e:STM}.
The bound on  $\| S(t;r) \|_2$ easily follows.   
 \end{proof}
 
The proof of the next result is contained in the Appendix.
\begin{lemma}
\label{t:StateTransIntBdd}
 For $t\ge 0$,
\begin{equation}
 \begin{aligned}
%  \int_0^t   S(t;r)   \xi(r)\, dr
% & =
%  \xi^I(t)  - S(t;0) \xi^I(0)
%  \\
%  &\quad  +  
%   \int_0^t \frac{1}{(1+r)} S(t;r)\barA \xi^I(r)\, dr
%   \\[1.5em]
   \int_0^t  & \frac{1}{1+r}  S(t;r)\barA \xi^I(r)\, dr
   \\
& =
 \frac{1}{1+t} \barA\xi^{I\!I}(t)  - \barA S(t;0) \barA\xi^{I\!I}(0)
 \\
 & \quad +  
  \int_0^t \frac{1}{(1+r)^2} S(t;r)[I+\barA] \barA \xi^{I\!I}(r)\,   dr\,  .
\end{aligned}
\label{e:Sint}
\end{equation}
There exists $b_\nu<\infty $ such that
\begin{equation}
   \int_0^t \frac{1}{(1+r)^2} \| S(t;r)\|\, dr \le b_\nu \frac{1}{ (1+t)^{\delta_S}}\,,\quad t\ge 0\, .
   \label{e:intStateTransm}
 \end{equation} 
 \qed
 \end{lemma}

\paragraph*{Proof of \Theorem{t:var}}
Lemmas~\ref{t:STM} and \ref{t:StateTransIntBdd}   give 
\[
\begin{aligned}
  \tilnu(t)  
  & =  S(t,0)\tilvarqsa(0) 
    +\clE_{\tilnu}(t)
    \\
\clE_{\tilnu}(t) &=     \frac{1}{t+1} \barA  \xi^{I\!I}(t)  
-  S(t;0)\barA  \xi^{I\!I}(0)   
\\
&\quad
   +  \int_0^t \frac{1}{(1+r)^2} S(t;r)[I+\barA]\barA \xi^{I\!I}(r)\,   dr\,   .
\end{aligned}
\] 
The two lemmas   imply that $\| \clE_{\tilnu}(t)\| \le O\bigl( (1+t)^{-\delta_S} \bigr)$.  \qed

\def\varqsaPR{\varqsa^{\text{PR}}}

% \subsection{Polyak-Ruppert averaging}
 
% \rd{Too much, no.   I think we just list open problems in the conclusions.}
% Why restrict to a gain of $(t+1)^{-1}$?   

% \begin{theorem}
% \label{t:PR}
% $1/t$ rate of convergence!
%  \end{theorem}
 
%  \begin{proof}
%  Integration by parts once more.   I wonder if we can get such nice finite-$t$ bounds?
%  \end{proof}

\section{Conclusion}
\label{s:conclude}
%%%%%%%%%%%%%%%%%%%% 

% When possible, it is best to avoid the introduction of i.i.d.\ or Markovian disturbances in learning algorithms.   \rd{{\huge examples to make this point}
% \\
% Isn't this now clear from the essay in the introduction?  $1/t$ rather than $1/\sqrt{t}$ rate of convergence?   In ``stochastic optimization'' they get $1/t$ rate by changing the problem -- the rate of convergence is to an approximation to the original problem. 
% \\
% The $1/t$ in stochastic optimization refers to the rate for the value, not the parameters.
% }

While QSA can  result in significant improvement in convergence rate,    the results of \Section{s:QSA} demonstrate that QSA algorithms must be implemented with care.   If the gain does not satisfy the assumptions of \Theorem{t:var} then the rate of convergence can be \textit{slower} than obtained in an i.i.d.\ or Markovian setting.  

There are many interesting topics for future research:
\begin{romannum}
\item   Further work is required to extend \Theorem{t:var} to the nonlinear algorithm.    

% \item
% Many of the results of this paper can extended to Polyak-Ruppert averaging.   
% In these algorithms the gain is increased to obtain estimates $ \bfvarqsa $ with high volatility, and these estimates are then averaged 
% \[
% \ddt \varqsaPR(t) = \frac{1}{t+1}[-\varqsaPR(t) +\varqsa(t)]
% \]
% A typical example is the gain $a(t) = (1+t)^{-\varrho}$ to generate $\bfvarqsa$,  with $\varrho\in (0,1)$.  The key question is variance.  Based on existing results for stochastic approximation  \cite{rup88,pol90,poljud92,kontsi04}, it may be expected that  (A6i) can be relaxed to $\Real(\lambda)<0$ (the matrix $A$ is merely Hurwitz),  and we obtain 
% \[
% \varqsaPR(t) = \varqsa^* + a(t)  \xi^I(t)  + o(a(t))
% \]

\item    Constant-gain algorithms are amenable to analysis using similar techniques.

\item We are most interested in applications to control and optimization:  
\item[(a)]
On-line learning applications, in which the function $f$ itself varies with time.   
That is, \eqref{eq:algo} is replaced by
\[
\ddt\varqsa(t) = a f_t(\varqsa(t),\xi(t)) \,,
\]
Analysis will be far simpler than in a traditional SA setting. 
\item[(b)]
Applications to decentralized control using reinforcement learning techniques.   In the LQR setting, the architecture for Q-learning or fixed-policy 
Q-learning might be informed by recent computational techniques for  control synthesis \cite{dhikhojovTAC18}.

\end{romannum}
%\end{romannum}

% \bigskip

%  {}
% \null  
% \null  %Needed with \usepackage{flushend}

\appendix
% \begin{center}
% {\LARGE \bf  Appendices}
% \end{center}

\section{Stability and Convergence}

%\rd{Why does Wikipedia delete Bellman?  Need to search history}

Many of the results that follow are based on the following standard inequality  -- common in the theory of stochastic approximation, as well as ordinary differential equations.

\begin{proposition}[Gronwall Inequality]   
\label{t:Gronwall}
Suppose that $\beta>0$,  $\bfalpha$ is non-decreasing,  and the scalar function of time $\bfmz$ satisfies the following   inequality on a time interval $[0,T]$:
\[
0\le  z(t)\leq \alpha (t)+\int _0^{t}\beta (s)z(s)\,  ds\,,  \qquad 0\le t\le T\, .
\]
Then,  $ z(t)\leq \alpha (t) e^{\beta t} $ for all $0\le t\le T$.
\qed
\end{proposition}

The next result is commonly used to provide bounds on solutions to differential equations.
\begin{lemma}
\label{t:ODEbdd}
Under the Lipschitz conditions in (A4) there is a non-decreasing, non-negative function $b_L$ such that
\[
\begin{aligned}
\| \varscaled(u+v) - \varscaled(u) \| 
  &\le b_L(v)  (1+\|\varscaled(u)\| )v
  \\
  \| \varode^u(u+v) - \varscaled(u) \| 
  &\le b_L(v)  (1+\|\varscaled(u)\| )v\,,\qquad u\ge 0\,, 
\end{aligned}
\]
where $ \varode^u$ is defined in 
\eqref{eq:init}.
\qed
\end{lemma}

\subsection{Ultimate Boundedness}

Ultimate boundedness required in the proof of Theorem \ref{t:convergence} is established in the following.
\begin{proposition}
\label{t:stability}
The solution to \eqref{eq:algoscaled} is ultimately bounded: there exists $b<\infty$ such that for any $\varscaled(0)=\varscaled$, $\limsup_{u \rightarrow \infty} \|\varscaled(u)\|\le b$.
\qed
\end{proposition}

Ultimate boundedness follows from a `drift condition' similar to (A2):
\begin{lemma}
\label{thm:drift}
The solution to \eqref{eq:algoscaled} is ultimately bounded if,
for some $T>0$, $0<\delta<1$,
and $u_0,b<\infty$, 
\begin{align*}
V(\varscaled(u+T)) - V(\varscaled(u)) &\le - \delta  \|\varscaled(u)\|\, ,
\end{align*}
for all $u \ge u_0,\ \|\varscaled(u)\|>b$.
\qed\end{lemma}

 \begin{proof}
For each initial condition  $\varscaled(0)=\varqsa$ and  $u\ge u_0$,
  denote
\[
\tau = \tau(\varqsa,u) \eqdef  \min(v\ge 0: \|\varscaled(u+v)\|\le b)
\]
where $u_0$ and $b$ are defined in the lemma.  If $ \|\varscaled(u)\|\le b$ then $\tau=0$,  and  if $\|\varscaled(u+v)\|>b$ for all $v\ge 0$, set $\tau =\infty$. 
For $m\in\mathbb{N}$, define $\tau_m = \min\{\tau,m\}$. Then,
\begin{align*}
-\tau_m b \delta &\ge - \delta \int_{u}^{u+\tau_m} \|\varscaled(w)\| \, dw
\\
&\ge  \int_{u}^{u+\tau_m} (V(\varscaled(w+T)) - V(\varscaled(w)))\, dw
\\
&= \int_{u+ \tau_m}^{u+\tau_m +T} V(\varscaled(w))\, dw -  \int_{u}^{u+T} V(\varscaled(w))\, dw 
\\
&\ge  - \int_{u}^{u+T} V(\varscaled(w))\, dw.
\end{align*}
The right hand side is independent of $m$, which establishes the upper bound
\[
\tau\le \frac{1}{ b \delta} 
 \int_{u}^{u+T} V(\varscaled(w))\, dw.
\]
Under the Lipschitz assumption on $V$, \Lemma{t:ODEbdd} can be applied to establish that for some finite constant $b_V$
\[
\tau\le b_V (1+\|\varscaled(u)\|)
\]
Hence $\tau(\varqsa,u)$ is everywhere finite.  

The uniform bound is used to construct a compact set $S$  that is \textit{absorbing}:   If $u\ge u_0$ and $\varscaled(u)\in S$,  then $\varscaled(u+v)\in S$ for $v\ge 0$.   Define this set to be a closed ball $S = \{ \varqsa : \|\varqsa\|\le b_1$,  where
\begin{align*}
    b_1 = \sup \{ \| \varscaled(u+v) \| :   u\ge u_0\,, \
					 v\le \tau(\varqsa,u)\, ,   \\
					 \|\varscaled(u)\|\le b+1\,, \  \varscaled(0)=\varqsa\in\Re^d \}
\end{align*}
That is,  $b_1$ bounds the maximum norm of an excursion of $\bfvarqsa$ after leaving the smaller set $\{ \varqsa : \|\varqsa\|\le b+1\}$. 
%
%\rd{nuts!
%This establishes a uniform bound on $\tau_m$ for all $m\in\mathbb{N}$, thus proving $\tau<\infty$. Now, using Gronwall's inequality (see \cite{hirsmadev04}) and the Lipschitz property of $f$, we can obtain the following bound on the growth of $\bfvarqsa$: %(see Appendix, \Section{sec:growthbnd})
%For some $b_1<\infty$,
%\[\|\varscaled(u+v) - \varscaled(u)\| \le b_1u\|\varscaled(u)\| \le b_1\tau\|\varscaled(u)\| \ \ \textrm{for }u\le\tau.\]
%Thus, we have $\|\varscaled(u)\| \le b(1+b_1\tau),\ \ u\ge u_0.$ 
%This proves ultimate boundedness of $\bfvarqsa$.}
\end{proof}

The proof of Proposition \ref{t:stability} is then obtained by establishing the conditions of Lemma \ref{thm:drift}.     To do so,    trajectories of $\bfvarscaled$ are compared with those of $\bfvarode$. Recall the definition of $ \varode^u$ is defined in 
\eqref{eq:init}. We have the suggestive representations:
\begin{equation}
\begin{aligned}
\varscaled(u+v) &=\varscaled(u)+\int_u^{u+v}  f(\varscaled(w),  \xi(g^{-1}(w)) ) \, dw
\\
\varode^u(u+v) &=\varscaled(u)+\int_u^{u+v}  \barf(\varode^u(w))  \, dw \,,\qquad u,v\ge 0\, .
\end{aligned}
\label{e:thetaComparisons}
\end{equation}

A Law of Large Numbers (LLN) is obtained  for the time scaled process $\{\xi(g^{-1}(u))\}_{u\ge 0}$. Notice the difference with a conventional LLN. Here, the interval of integration is some arbitrary fixed $T$, and the averaging becomes more accurate as the interval is shifted towards infinity.  The proof is given in the Appendix.  

\begin{lemma}
\label{thm:LLN}
For any $u,T>0$, $\|\varqsa\|\geq 1$, the function $f$ satisfies the following bound:
 \begin{equation}
 \label{eq:LLN}
 \left\|\frac{1}{T}\int_u^{u+T} f(\varqsa,\xi(g^{-1}(w)))\, dw \ -\  \bar{f}(\varqsa)\right\| \le \epsy_f(u)\|\varqsa\|/T,
 \end{equation}
where $\epsy_f(u)\rightarrow 0$ as $u\rightarrow\infty$.
\end{lemma}

\begin{proof}
Denote $\tilf(\varqsa,\xi(w))=f(\varqsa,\xi(w))$ for each $w$ and $\varqsa$, and
\[
\clE_f(\varqsa,t) \eqdef  \frac{1}{t}\int_0^t \tilf(\varqsa,\xi(w))\, dw \, .
\]
By assumption (A3), for $\|\varqsa\|\ge 1$,
\begin{equation}
\label{eq:fhatbnd}
\|\clE_f(\varqsa,t)\| \le b_0(1+\|\varqsa\|)/t\le 2b_0\|\varqsa\|/t.
\end{equation}
The following integral is simplified using integration by parts:
\begin{align*}
\int_{t_0}^{t_1} a(t) \tilf(\varqsa,\xi(t))\, dt 
&= \left[a(t)\int_0^t \tilf(\varqsa,\xi(r))\, dr\right]_{t_0}^{t_1} 
\\
&\ \ \ - \int_{t_0}^{t_1} a'(t) \int_0^t \tilf(\varqsa,\xi(r))\, drdt
\\
&= [t_1a(t_1)\clE_f(\varqsa,t_1) - t_0a(t_0)\clE_f(\varqsa,t_0)]
\\
&\ \ \  - \int_{t_0}^{t_1} ta'(t)\clE_f(\varqsa,t)\, dt
\end{align*}
Rearranging and taking norms, we obtain on applying \eqref{eq:fhatbnd}  
\[
\left\|\int_{t_0}^{t_1} a(t)\tilf(\varqsa,\xi(t))\, dt  \right\| 
	\le 4b_0a(t_0)\|\varqsa\|.
\]
We have used the fact that $a(t)$ is non-increasing.
Letting $t_0=g^{-1}(u),\ t_1=g^{-1}(u+T)$ and $t=g^{-1}(w)$ yields
\begin{align*}
\left\|\frac{1}{T}\int_u^{u+T} f(\varqsa,\xi(g^{-1}(w)))\, dw \ -\  \bar{f}(\varqsa)\right\| \\
\qquad \le 4b_0 a(g^{-1}(u))\|\varqsa\|/T.
\end{align*}
Set $\epsy_f(u) \eqdef 4b_0  a(g^{-1}(u))$. As $u\rightarrow\infty$, $g^{-1}(u)\rightarrow\infty$ and hence, $\epsy_f(u)\rightarrow 0$. This completes the proof.
\end{proof}

The next lemma bounds the difference between $\bfvarscaled$ and
$\bfvarode^u$. This bound is then used to establish a drift condition for $\bfvarscaled$.   

\begin{lemma}
\label{thm:growthbnd}
For some $\bar{b}<\infty$ and any $0 < T\le 1$, 
\begin{equation}
\label{eq:growthbnd}
\|\varscaled(u+T) - \varode^u(u+T)\| \le( \bar{b} T^2 + e^L \epsy_f(u))\|\varscaled(u)\|\,
\end{equation}
whenever $\|\varscaled(u)\|\ge 1$, where $\epsy_f(u)$ is given by Lemma \ref{thm:LLN} and $L$ is given by Assumption (A4).
 \qed
\end{lemma}

\begin{proof}
Denote $\clE^u(w) =  \varode^u(w) - \varscaled(w) $ for $w\ge u$.   The pair of identities  \eqref{e:thetaComparisons} give 
\begin{equation}
\begin{aligned}
\clE^u(u+v) &=  \int_u^{u+v}  [ \barf(\varscaled(w)) )   - f(\varscaled(w),  \xi(g^{-1}(w)) ) ] \, dw
\\
 &\quad +   \int_u^{u+v} [ \barf(\varode^u(w)) ) -\barf(\varscaled(w)) ) ]\, dw \,, \, \, u,v\ge 0\, .
\end{aligned}
\label{e:clEu}
\end{equation} 
The   Lipschitz conditions in (A4)  is used to  bound the integrands:
\[
\begin{aligned}
&\|  \barf(\varscaled(w)) )   - \barf(\varscaled(u)) )     \| \le  L   \| \varscaled(w) - \varscaled(u) \|
\\
&\|    f(\varscaled(w),  \xi(g^{-1}(w)) )    - f(\varscaled(u),  \xi(g^{-1}(w)) )  \|   \\
& \qquad \qquad \le  L   \| \varscaled(w) - \varscaled(u) \|
 \\
&\|  \barf(\varode^u(w)) ) -\barf(\varscaled(w)) \|  \le \|\clE^u(w)\| 
\end{aligned}
\]
Consequently, for any $0<T\le 1$ and any
$0\le v\le T$,
\[
\begin{aligned}
\| \clE^u(u+v)  \|&\le \Bigl\|   \int_u^{u+v}  [ \barf(\varscaled(u)) )   - f(\varscaled(u),  \xi(g^{-1}(w)) ) ] \, dw \Bigr\| 
\\
 &\quad + 2  L \int_u^{u+v}  \| \varscaled(w) - \varscaled(u) \| \, dw
 \\
 &\quad + L  \int_u^{u+v} \| \clE^u(w)  \| \, dw \,,\qquad u,v\ge 0\, .
\end{aligned}
\]
Applying \Lemma{thm:LLN}, when $\|\varscaled(u)\|\ge 1$,
\[
\begin{aligned}
\| \clE^u(u+v)  \|& \le  L  \int_u^{u+v} \| \clE^u(w)  \| \, dw   + \epsy_s^u(v)
\\
 \epsy_s^u(v)
& \le \epsy_f(u)\|\varscaled(u)\| +  2  L \int_u^{u+v}  \| \varscaled(w) - \varscaled(u) \| \, dw\, .
\end{aligned}
\]
\Lemma{t:ODEbdd} then gives  $ \epsy_s^u(v) \le  \epsy_f(u)\|\varscaled(u)\| +   L b_L(T)  (1+\|\varscaled(u)\| ) v^2$.    Applying
 Gronwall's Lemma, for any $0<T\le 1$  
\[
\| \clE^u(u+T)  \| \le    \Bigl[ \epsy_f(u)\|\varscaled(u)\| +   L b_L(1)  (1+\|\varscaled(u)\| ) T^2 \Bigr]  e^{L}.
\] 
 This
  completes the proof.
\end{proof}

\textit{Proof of Proposition \ref{t:stability}}
Recall that $V$ is the Lyapunov function  and $c_0>0$ is the constant introduced in (A2). 
For $0\le T \le 1$, $\|\varscaled(u)\|\ge c_0+1$,
\begin{align*}
&V(\varscaled(u+T)) - V(\varscaled(u))  = V(\varscaled(u+T))\\
&\qquad - V(\varode^u(u+T))   + V(\varode^u(u+T)) - V(\varode^u(u)) 
		\\ 
&\, \le |V(\varscaled(u+T)) - V(\varode^u(u+T))| \\
&\qquad + V(\varode^u(u+T)) - V(\varode^u(u))\\
&\, \le L\|\varscaled(u+T) - \varode^u(u+T)\|  - T\|\varscaled(u)\|\\
&\, \le L( \bar{b}T^2 + e^L\epsy_f(u))\|\varscaled(u)\| - T\|\varscaled(u)\|,
\end{align*}
where the second inequality follows from the Lipschitz assumption on $V$ and the last inequality uses \Lemma{thm:growthbnd}.  Let us choose $T>0x$ small enough to make $2L \bar{b} T^2 \le T/2$, and then $u_0$ large enough so that $e^L \epsy_f(u)\le \bar{b} T^2$ for all $u \ge u_0$, which leads to
\[
	V(\varscaled(u+T)) - V(\varscaled(u))\le -\frac{T}{2} \|\varscaled(u)\|.
\]
\Lemma{thm:drift} completes the proof. \qed

\subsection{Proof of Lemma \ref{thm:limsuplimsup}}

The first step in the proof  is a  variation of \Lemma{thm:LLN}.  
\begin{lemma}
\label{thm:limsup1}
Assume that $\bfvarscaled$ is bounded. 
Then, for any $T>0$,
\[
	\lim_{u\rightarrow\infty} \sup_{v\in[0,T]}\Bigl\|
		\int_u^{u+v} \bigl[ f(\varscaled(w),\xi(g^{-1}(w))) - \bar{f}(\varscaled(w) ) \bigr]\, dw
							\Bigr \| = 0.
\]
\qed
\end{lemma} 

\begin{proof}
Denote %$e_f(w) = f(\varscaled(w),\xi(g^{-1}(w))) - \bar{f}(\varscaled(w) )$,  and
\[
\clE_{e_f}^u(u+v) = \int_u^{u+v}  [f(\varscaled(w),\xi(g^{-1}(w))) - \bar{f}(\varscaled(w) )] \, dw
\]
The lemma states that this converges to zero  as $u\to\infty$, uniformly for $v$ in bounded intervals, provided $\bfvarqsa$ is bounded. 

To prove the assertion, fix $\delta>0$ and denote $u_k = u+k\delta$ for $k\ge 0$.   
As in the theory of Riemannian integration, the Lipschitz conditions in (A4) imply the following bound: 
\begin{align*}
&\clE_{e_f}^u(u+v)   \\
& \, =     \sum_{k=0}^{n_v-1}  	
		 \int_{u_k}^{u_k+\delta } \bigl[ f(\varscaled(u_k),\xi(g^{-1}(w))) - \bar{f}(\varscaled(u_k) ) \bigr]\, dw  +   \epsy_{e_f}^u     
\end{align*}
where $n_v$ denotes the integer part of $v/\delta$,  and $\| \epsy_{e_f}^u\|  \le  b_L v \delta$ 
for some constant $b_L<\infty$.  The bound is uniform in $u$ under the assumption that $\bfvarqsa$ is bounded. 

\Lemma{thm:LLN} and the triangle-inequality imply the bound
\[
\| \clE_{e_f}^u(u+v)  \|  \le     \sum_{k=0}^{n_v-1}  	
\epsy_f(u_k)\|\varscaled(u_k)\|     +   b_L v \delta 
\]
Let $b_\varqsa<\infty$ denote a constant satisfying $\|\varscaled(u)\|\le b_\varqsa$ for all $u$.  Then, 
\[
\| \clE_{e_f}^u(u+v)  \|  \le    b_\varqsa \frac{v}{\delta}   \sup_{u'\ge u} \epsy(u')   +   b_L v \delta 
\]
\Lemma{thm:LLN} then implies that 
for any $T>0$,
\[
	\limsup_{u\rightarrow\infty} \sup_{v\in[0,T]}\| \clE_{e_f}^u(u+v)  \|  \le b_L T \delta 
\]
This completes the proof, since $\delta>0$ was arbitrary.
\end{proof}

The next result is very similar to Lemma~1 in Chapter~2 of \cite{bor08a}.  
\begin{lemma}
\label{thm:limsup}
Assume that $\bfvarscaled$ is bounded.   Then, for any $T>0$, 
\[
	\lim_{u\rightarrow\infty} \sup_{v\in[0,T]}\|\varscaled(u+v) - \varode^u(u+v)\| = 0.
\]
\qed
\end{lemma}

\begin{proof}
This result is a refinement of \Lemma{thm:growthbnd}, and its proof begins with the representation  
\eqref{e:clEu} for  $\clE^u(w) =  \varode^u(w) - \varscaled(w) $,  $w\ge u$.

The Lipschitz conditions in (A4)  imply the   bound:
\[
|\clE^u(u+v) | \le  \delta^u 
			+ L \int_u^{u+v}  |\clE^u(w)|  \, dw 
\]
where 
\[
\delta^u \eqdef 
\inf_{u'\ge u}
\max_{0\le v\le T} \Big\|
 \int_{u'}^{u'+v}  [ \barf(\varscaled(w)) )   - f(\varscaled(w),  \xi(g^{-1}(w)) ) ] \, dw \Bigr\|
\]
\Prop{t:Gronwall} then gives $|\clE^u(u+v) | \le e^{LT} \delta^u$ for all $u$, and all $0\le v\le 1$.  
The  error term $\delta^u$ vanishes as $u\to\infty$ due to \Lemma{thm:limsup1}.
\end{proof}

\subsection{Proof of Lemma \ref{t:StateTransIntBdd}}

The bound  \eqref{e:intStateTransm} follows from \Lemma{t:STM} and elementary calculus.  The representation of the integral requires further work.

Let $\bfgamma$ denote a continuous vector-valued function of time.  We prove that for any $m\ge 0$,
\begin{equation}
\begin{aligned}
   &\int_0^t \frac{1}{(1+r)^m} S(t;r) \gamma(r)\, dr
 =
 \frac{1}{(1+t)^{m}} \gamma^{I}(t)  - S(t;0) \gamma^{I}(0)
 \\
 & \quad +  
  \int_0^t \frac{1}{(1+r)^{m+1}} S(t;r)[mI+\barA] \gamma^{I}(r)\,   dr\,,\quad t>0\, .
\end{aligned}
\label{e:genSint}
\end{equation}
The identity \eqref{e:Sint} will immediately follow, using $m=1$ and $\gamma(r) = \barA\xi^I(r)$, $r\ge 0$.

The following identity follows from the state transition matrix property that $I = S(t;r) S(r;t) $:
\[
\ddr S(t;r) = - \frac{1}{r+1} S(t;r) \barA
\]
Consequently, for each $m\ge 0$,
\begin{equation}
\ddr  \frac{1}{(1+r)^m}S(t;r)  
=
- \frac{1}{(1+r)^{m+1}}S(t;r)[mI+\barA]
  \label{e:mSint}
\end{equation}

Apply integration by parts:
 \[
 \begin{aligned}
 &\int_0^t \frac{1}{(1+r)^m} S(t;r) \gamma(r)\, dr
 =
 \int_0^t \frac{1}{(1+r)^m} S(t;r)  \frac{d}{dr}\gamma^I(r)\, dr
 \\
 &\qquad \qquad =  
 \frac{1}{(1+r)^m} S(t;r) \gamma^I(r) \Big|_0^t
 \\
 & 
 \qquad \qquad \qquad - \int_0^t\frac{d}{dr} \Bigl( \frac{1}{(1+r)^m}S(t;r) \Bigr) \gamma^I(r)\, dr
\end{aligned}
 \] 
Substituting   \eqref{e:mSint}  completes the proof of \eqref{e:genSint}. \qed
%
%
%\begin{proof}[Proof of \Prop{t:rho-a}]
%Following the proof of \Theorem{t:var}, denote $\nu(t)=(t+1)^\varrho\tilvarqsa(t)$ and differentiate:  
%\[
%\begin{aligned}
%\ddt \nu(t) & = \varrho (t+1)^{\varrho-1}\tilvarqsa(t) +   A\tilvarqsa(t) + \xi(t)  
%\\
%& = \bigl(  \varrho (1+t)^{-1} +(1+t)^{-\varrho}\bigr) A \nu(t) + \xi(t)  
%\end{aligned}
%\]
%\bl{...to be filled in ...}
% \end{proof}

\bibliographystyle{IEEEtran}
\bibliography{IEEEabrv,strings,markov,q,QSAextras,bandits}

 {}
\null  
\null  %Needed with \usepackage{flushend}

\end{document}